\documentclass[a4paper]{amsart}

\usepackage[english]{babel}
\usepackage[utf8x]{inputenc}
\usepackage[T1]{fontenc}

\usepackage[a4paper,top=3cm,bottom=2cm,left=3cm,right=3cm,marginparwidth=1.75cm]{geometry}

\usepackage{soul}

\usepackage[dvipsnames]{xcolor}

\usepackage[
colorlinks=true,
linkcolor=blue,
citecolor=blue,
urlcolor=blue,
]{hyperref}

\usepackage{amsmath, amsthm, amssymb,esint,mathrsfs}
\usepackage{graphicx}
\usepackage[colorinlistoftodos]{todonotes}

\newtheorem{theorem}{Theorem}[section]
\newtheorem{lemma}[theorem]{Lemma}
\newtheorem{conjecture}[theorem]{Conjecture}

\theoremstyle{definition}
\newtheorem{definition}[theorem]{Definition}
\newtheorem{remark}[theorem]{Remark}

\numberwithin{equation}{section}

\DeclareMathOperator*{\esssup}{ess\,sup}
\DeclareMathOperator*{\essinf}{ess\,inf}

\newcommand{\vertiii}[1]{{\left\vert\kern-0.25ex\left\vert\kern-0.25ex\left\vert #1 
    \right\vert\kern-0.25ex\right\vert\kern-0.25ex\right\vert}}

\title{Sawyer-type inequalities for Lorentz spaces}

\begin{document}

\author[C. P\'erez]{Carlos P\'erez$^*$}
\address{Department of Mathematics, University of the Basque Country, Ikerbasque
and BCAM, Bilbao, Spain.}
\email{cperez@bcamath.org}

\author[E. Roure-Perdices]{Eduard Roure-Perdices$^{**}$}
\address{Departament de Matem\`atiques i Inform\`atica, Universitat de Barcelona, 08007 Barcelona, Spain.} 
\email{eduardroure@protonmail.ch}

\thanks{C. P. was supported by grants MTM2017-82160-C2-2-P (MINECO) and BCAM Severo Ochoa accreditation \\ \indent SEV-2017-0718 (MINECO). \\ \indent E. R. P. was supported by grant FPU14/04463 (MECD). \\ 
\indent\emph{E-mail addresses:} $^{*}$\texttt{cperez@bcamath.org}, $^{**}$\texttt{eduardroure@protonmail.ch}}

\subjclass[2010]{42B25, 46E30}

\keywords{Lorentz spaces, Sawyer-type inequalities, $A_p^{\mathcal R}$ and $A_{\vec P}^{\mathcal R}$ weights}

\begin{abstract}
The Hardy-Littlewood maximal operator satisfies the classical Sawyer-type estimate
$$
\left \Vert \frac{Mf}{v}\right \Vert_{L^{1,\infty}(uv)} \leq C_{u,v} \Vert f \Vert_{L^{1}(u)},
$$
where $u\in A_1$ and $uv\in A_{\infty}$.
We prove a novel extension of this result to the general restricted weak type case. That is, for $p>1$, $u\in A_p^{\mathcal R}$, and $uv^p \in A_\infty$,
$$
\left \Vert \frac{Mf}{v}\right \Vert_{L^{p,\infty}(uv^p)} \leq C_{u,v} \Vert f \Vert_{L^{p,1}(u)}.
$$

From these estimates, we deduce new weighted restricted weak type bounds and Sawyer-type inequalities for the $m$-fold product of Hardy-Littlewood maximal operators. We also present an innovative technique that allows us to transfer such estimates to a large class of multi-variable operators, including $m$-linear Calder\'on-Zygmund operators, avoiding the $A_\infty$ extrapolation theorem and producing many estimates that have not appeared in the literature before. In particular, we obtain a new characterization of $A_p^{\mathcal R}$. 

Furthermore, we introduce the class of weights that characterizes the restricted weak type bounds for the multi(sub)linear maximal operator $\mathcal M$, denoted by $A_{\vec P}^{\mathcal R}$, establish analogous bounds for sparse operators and m-linear Calder\'on-Zygmund operators, and study the corresponding multi-variable Sawyer-type inequalities for such operators and weights. 

Our results combine mixed restricted weak type norm inequalities, $A_p^{\mathcal R}$ and $A_{\vec P}^{\mathcal R}$ weights, and Lorentz spaces.
\end{abstract}

\maketitle

\section{Introduction}

``Sawyer-type inequalities'' is a terminology coined in the paper \cite{CUMP}, where the authors prove that if $u\in A_1$, and $v\in A_1$ or $uv\in A_\infty$, then
\begin{equation}\label{eqclassic}
    uv \left(\left\{x\in \mathbb R^n:  \frac{|T(fv)(x)|}{v(x)}> t  \right\} \right) \leq \frac{C}{t} \int_{\mathbb R^n} |f(x)|u(x)v(x)dx, \quad t>0,
\end{equation}
where $T$ is either the Hardy-Littlewood maximal operator or a linear Calder\'on-Zygmund operator. This result extends some questions previously considered by B. Muckenhoupt and R. Wheeden in \cite{MW}, and solves in the affirmative a conjecture formulated by E. Sawyer in \cite{Sa} concerning the Hilbert transform. These problems were advertised by B. Muckenhoupt in \cite{M}, where the terminology ``mixed type norm inequalities'' was introduced and was also used since then in other papers like \cite{AM} or \cite{MOS}. In general, this terminology refers to certain weighted estimates for some classical operators $T$, where a weight $v$ is included in their level sets; that is, 
\begin{equation}\label{LevelSet}
\left\{x\in \mathbb R^n:  \frac{|Tf(x)|}{v(x)}> t  \right\}, \quad t>0.
\end{equation}
The structure of such sets makes it impossible, or very difficult, to use classical tools to measure them, such as Vitali's covering lemma or interpolation theorems. 

In this paper, we consider mixed restricted weak type norm inequalities, or Sawyer-type inequalities for Lorentz spaces; that is, we study estimates of the form
\begin{equation}\label{ModelEx}
w \left(\left\{x\in \mathbb R^n:  \frac{|Tf(x)|}{v(x)}> t  \right\} \right)^{1/p} \leq \frac{C}{t} \Vert f \Vert_{L^{p,1}(u)}, \quad t>0,
\end{equation}
where $p\geq1$, $T$ is a classical operator, and $u,v,w$ are weights. We also consider extensions of such inequalities to the multi-variable setting. Our goal is to prove estimates like \eqref{ModelEx} for sub-linear and multi-sub-linear maximal operators, and multi-linear Calder\'on-Zygmund operators. Observe that in the classical situation, namely when $u=w$, $v\approx 1$, and $T$ is either the Hardy-Littlewood maximal operator or a linear Calder\'on-Zygmund operator, the inequality (\ref{ModelEx}) holds if $w\in A_p^{\mathcal R}$ (some authors use the notation $A_{p,1}$ for this class of weights, as in \cite{CHK}). The case when $v\not \approx 1$ is much more difficult, and in what follows, we will study it in great detail. 

The starting point of this paper and our primary motivation to consider Sawyer-type inequalities for Lorentz spaces comes from the study of the $m$-fold product of Hardy-Littlewood maximal operators, 
$$
M^{\otimes}(\vec f)(x):= \prod_{i=1}^m Mf_i(x), \quad x\in \mathbb R^n.
$$
M. J. Carro and E. R. P. proved in \cite{CR} that for exponents $1\leq p_1,\dots,p_m<\infty$ and $\frac{1}{p}=\frac{1}{p_1}+\dots+\frac{1}{p_m}$, and weights $w_1,\dots,w_m$ in $A_\infty$ and $\nu_{\vec w}:=w_1^{p/p_1} \dots w_m^{p/p_m}$, a necessary condition to have 
\begin{equation}\label{eqprodhl}
    M^{\otimes}:L^{p_1,1}(w_1)\times \dots \times L^{p_m,1}(w_m) \longrightarrow L^{p,\infty}(\nu_{\vec w})
\end{equation}
is that $w_i \in A_{p_i}^{\mathcal R}$, for $i=1,\dots,m$. They left as an open question to prove that this last condition is also sufficient for (\ref{eqprodhl}) to hold. It is reasonable to think that this may indeed be true since the endpoint case was proved in \cite{LOPTT}. That is, for weights $w_1,\dots,w_m \in A_1$, we have that 
\begin{equation}\label{eqendpoint}
M^{\otimes}:L^{1}(w_1)\times \dots \times L^{1}(w_m) \longrightarrow L^{1/m,\infty}(w_1^{1/m} \dots w_m^{1/m}).
\end{equation}
To prove this result, one has to control the following quantity for $t>0$, which is related to the level sets (\ref{LevelSet}):
$$
\nu_{\vec w}\left(\left \{x\in \mathbb R^n : M^\otimes(\vec f)(x)>t \right \}\right)=\nu_{\vec w}\left(\left \{x\in \mathbb R^n : Mf_i(x)>\frac{t}{\prod_{j\neq i} Mf_j(x)} \right \}\right),
$$
where $\nu_{\vec{w}}=w_1^{1/m} \dots w_m^{1/m}$. This is achieved by applying the classical Sawyer-type inequality (\ref{eqclassic}) for the Hardy-Littlewood maximal operator $M$ in combination with the observation that for locally integrable functions $h_1,\dots,h_k$, $\prod_{j=1}^k (Mh_j)^{-1}\in RH_\infty$, with constant depending only on $k$ and the dimension $n$.

As we will show in Theorem~\ref{prodhl}, it turns out that the bound (\ref{eqprodhl}) holds if $w_i \in A_{p_i}^{\mathcal R}$, for $i=1,\dots,m$, solving in the affirmative the open question in \cite{CR} and completing the characterization of the restricted weak type bounds of $M^\otimes$ for $A_\infty$ weights. The strategy that we follow is similar to the one in \cite{LOPTT} for the endpoint case (\ref{eqendpoint}), but we have to replace the classical Sawyer-type inequality (\ref{eqclassic}) by the estimate obtained in Theorem~\ref{sawyer}, which is a new restricted weak Sawyer-type inequality involving the class of weights $A_p^{\mathcal R}$. That is, 
\begin{equation}\label{eqsawyer}
    \left \Vert \frac{Mf}{v}\right \Vert_{L^{p,\infty}(uv^p)} \leq C_{u,v} \Vert f \Vert_{L^{p,1}(u)},
\end{equation}
for $p>1$, $u\in A_p^{\mathcal R}$, and $uv^p\in A_\infty$. The $A_p^{\mathcal R}$ condition on the weight $u$ is a natural assumption since it is necessary when $v \approx 1$. In Lemma~\ref{creciente} we also manage to track the dependence of the constant $C_{u,v}$ on the weights $u$ and $uv^p$, even in the endpoint case $p=1$, refining the bound (\ref{eqclassic}) in \cite{CUMP}.

There is no reason to restrict ourselves to the study of one-variable Sawyer-type inequalities. Quite recently, the bound (\ref{eqclassic}) has been extended to the multi-variable setting in \cite{kob}. More precisely, for weights $w_1,\dots,w_m\in A_1$, and $v\in A_\infty$,
\begin{equation}\label{eqmultiend}
    \left \Vert \frac{\mathcal M (\vec f)}{v} \right \Vert_{L^{1/m,\infty}(\nu_{\vec w}v^{1/m})} \leq \left \Vert \frac{\prod_{i=1}^m Mf_i}{v} \right \Vert_{L^{1/m,\infty}(\nu_{\vec w}v^{1/m})} \lesssim \prod_{i=1}^m \Vert f_i \Vert_{L^1(w_i)}.
\end{equation}
Inspired by this result, we follow a similar approach to extend our Sawyer-type inequality (\ref{eqsawyer}) to the multi-variable setting, obtaining a generalization of (\ref{eqmultiend}) in Theorem~\ref{msawyer}. That is, for weights $w_1,\dots,w_m$ and $v$ such that for $i=1,\dots,m$, $w_i\in A_{p_i}^{\mathcal R}$ and $w_iv^{p_i} \in A_\infty$, 
\begin{equation}\label{eqmulti}
    \left \Vert \frac{\mathcal M (\vec f)}{v} \right \Vert_{L^{p,\infty}(\nu_{\vec w}v^{p})} \leq \left \Vert \frac{\prod_{i=1}^m Mf_i}{v} \right \Vert_{L^{p,\infty}(\nu_{\vec w}v^{p})} \lesssim \prod_{i=1}^m \Vert f_i \Vert_{L^{p_i,1}(w_i)}.
\end{equation}
Observe that this result is an extension of (\ref{eqprodhl}). To our knowledge, this multi-variable mixed restricted weak type inequalities for maximal operators involving the $A_p^{\mathcal{R}}$ condition on the weights have not been previously studied, and we found no record of them being conjectured in the literature.

Motivated by the conjecture of E. Sawyer in \cite{Sa}, we can ask ourselves if it is possible to obtain bounds like (\ref{eqmulti}) for multi-linear Calder\'on-Zygmund operators $T$. Once again, the endpoint case $p_1=\dots=p_m=1$ has already been considered and extensively investigated in \cite{kob}. There, it was shown that for weights $w_1,\dots,w_m\in A_1$, and $\nu_{\vec w} v^{1/m} \in A_\infty$,
\begin{equation}\label{eqmultiendcz}
    \left \Vert \frac{T (\vec f)}{v} \right \Vert_{L^{1/m,\infty}(\nu_{\vec w}v^{1/m})} \lesssim \prod_{i=1}^m \Vert f_i \Vert_{L^1(w_i)},
\end{equation}
as a corollary of (\ref{eqmultiend}), combined with a result in \cite{ompe}, that allows replacing $\mathcal M$ by $T$ using an extrapolation type argument based on the $A_\infty$ extrapolation theorem obtained in \cite{CUMP3,GCMP}. We succeed in our goal and manage to get an extension of (\ref{eqmultiendcz}) to the general restricted weak setting. In Theorem~\ref{czo} we prove, among other things, that for weights $w_1,\dots,w_m$ and $v$ such that for $i=1,\dots,m$, $w_i\in A_{p_i}^{\mathcal R}$ and $w_iv^{p_i} \in A_\infty$, and some other technical hypotheses on the weights,
\begin{equation}\label{eqmulticz}
    \left \Vert \frac{T (\vec f)}{v} \right \Vert_{L^{p,\infty}(\nu_{\vec w}v^{p})} \lesssim \prod_{i=1}^m \Vert f_i \Vert_{L^{p_i,1}(w_i)}.
\end{equation}
To achieve this, we build upon (\ref{eqmulti}), but unlike in \cite{kob}, we manage to avoid the use of extrapolation arguments like the ones in \cite{ompe}. Instead, we present in Theorem~\ref{sparsemax} a novel technique that allows us to replace $\mathcal M$ by $T$ exploiting the fine structure of the Lorentz space $L^{p,\infty}$, the $A_p^{\mathcal R}$ condition, and the recent advances in sparse domination. 

One can even go further and consider inequalities like \eqref{eqmulticz} assuming multi-variable conditions on the weights involved, as it was done in \cite{kob} with the endpoint case $p_1=\dots=p_m = 1$ and weights in $A_{\vec 1}$. In Section~\ref{section5}, we discuss our findings on this matter. There, we introduce the class of weights that characterizes the restricted weak type bounds of $\mathcal M$, denoted by $A_{\vec P}^{\mathcal R}$, study some of its properties, deduce the corresponding restricted weak type bounds for sparse operators and multi-linear Calder\'on-Zygmund operators, and conjecture the main results on Sawyer-type inequalities with weights in $A_{\vec P}^{\mathcal R}$. It is worth mentioning that we couldn't find in the literature any trace of results like (\ref{eqmulticz}) involving $\mathcal M$ or multi-linear Calder\'on-Zygmund operators, $A_p^{\mathcal R}$ or $A_{\vec P}^{\mathcal R}$ weights, and mixed restricted weak type inequalities.

Curiously, we didn't find much about Sawyer-type inequalities for Lorentz spaces apart from the endpoint results studied in \cite{AM,CUMP,ekm,kop,kob,MOS,M,MW,ompe,Sa}, and some endpoint estimates for multi-variable fractional operators (see \cite{swyrbelen}), multi-linear pseudo-differential operators (see \cite{cxy}), and the Hardy averaging operator (see \cite{lomr,mro}). As we have seen before, these inequalities are fundamental to understand the behavior of the operator $M^\otimes$, but they appear naturally in the study of other classical operators, even in the one-variable case. Consider, for example, the case of the Hilbert transform $H$. Indeed, if $p>1$ and $w\in A_p^{\mathcal R}$, it is well-known that $H:L^{p,1}(w)\longrightarrow L^{p,\infty}(w)$. Hence, duality,  linearity, and self-adjointness of $H$ yield
$$
\left \Vert  \frac{H(fw)}{w}  \right \Vert_{L^{p',\infty}(w)} \leq C_w \left \Vert  f  \right \Vert_{L^{p',1}(w)}.
$$
This is an example of an estimate like (\ref{ModelEx}) involving the $A_p^{\mathcal R}$ condition on the weights and obtained almost without effort. The same inequality holds for the Hardy-Littlewood maximal operator $M$, but we cannot use the same argument, as shown in \cite{CS}. In Theorem~\ref{dualsawyer} we will generalize such a result for $M$, obtaining as a particular case, a new characterization of $A_p^{\mathcal R}$ and an alternative proof of the result in \cite{CS}. In \cite{hype,lernerdual}, one can find similar endpoint estimates for Calder\'on-Zygmund operators, with $p'=1$ and $w \in A_1$ (see also \cite{carr,opre,osra,stock,wwz}).

Sawyer-type inequalities also arise in the broadly studied topic involving commutators of linear operators $T$ with a $BMO$ function $b$, although we will not deal with them in this paper. The crucial initial observation is that we can write $[b,T]$ as a complex integral operator using Cauchy's integral theorem, obtaining that for $\varepsilon>0$, 
$$
[b,T]f
=\frac{1}{2\pi i}\int_{\{z\in\mathbb C: \, |z|=
\varepsilon\}} \frac{T_z(f)}{z^2}\,dz, 
$$
where 
$$T_z(f) := e^{zb} T\left(\frac{f}{e^{zb}}\right), \quad z \in \mathbb{C}. 
$$
This approach was introduced in the celebrated paper \cite{CRW} and was further developed  in \cite{AKMP}. In the context of Lorentz spaces, for $p>1$ and a weight $w$, and in virtue of Minkowski's inequality, we get that for any $\varepsilon>0$,
\begin{equation*}
    \left \Vert [b,T]f  \right \Vert_{L^{p,\infty}(w)} \leq \frac{1}{\varepsilon} \,\sup_{z\in \mathbb C: \, |z|=\varepsilon}  
\left \Vert T_z(f)  \right \Vert_{L^{p,\infty}(w)}.
\end{equation*}
Since $b\in BMO$, as a consequence of the John-Nirenberg inequality, there is a constant $s_0>0$ such
that for $|z|\leq s_0$, $v^{-1}:=|e^{zb}|=e^{\Re(z) b} \in A_p$, and hence, it is possible to deduce weighted inequalities for commutators from estimates of the form
\begin{equation*}
\left \Vert  \frac{T(fv)}{v}  \right \Vert_{L^{p,\infty}(w)} \lesssim \left \Vert  f  \right \Vert_{X},
\end{equation*}
for a norm or a quasi-norm $\Vert \cdot  \Vert_{X}$, and $v^{-1}\in A_p$. Further results for commutators involving Sawyer-type inequalities can be found in \cite{B,BCP} (see also \cite{BCP2,BCP3}). 

Recently, E. R. P. has shown in \cite{ropethesis} that Sawyer-type inequalities for Lorentz spaces play a fundamental role in the extension to the multi-variable setting of the restricted weak type Rubio de Francia's extrapolation presented in \cite{cgs,CS}. His approach suggests that Conjecture~\ref{conjmsawyer} will be crucial for proving multi-variable extrapolation theorems involving weights in $A_{\vec P}^{\mathcal R}$. 

\section{Preliminaries}

\subsection{Lorentz spaces and classical weights}

Let us recall the definition of the Lebesgue and Lorentz spaces (see \cite{BS}).  For $p>0$ and an arbitrary measure space $(X,\nu)$, $L^{p,1}(\nu)$ is the Lorentz space of $\nu$-measurable functions such that
$$
\Vert f \Vert_{L^{p,1}(\nu)}
:=p\int_0^\infty \lambda_f^\nu (y) ^{1/p}  dy  =\int_0^\infty  f_\nu^*(t) t^{1/p}  \frac{dt}{t} <\infty,
$$
$L^{p}(\nu)$ is the Lebesgue space of $\nu$-measurable functions such that
$$
\Vert f \Vert_{L^{p}(\nu)}
:=\left(\int_X |f|^p \nu \right)^{1/p}  <\infty \quad \text{(or } 
\nu-\esssup_X |f| < \infty \text{ if } p=\infty), 
$$
and $L^{p,\infty}(\nu)$ is the Lorentz space of $\nu$-measurable functions such that
$$
\Vert f \Vert_{L^{p,\infty}(\nu)}:= \sup_{y>0}y  \lambda_f^\nu(y)^{1/p}= \sup_{t>0}t^{1/p}f_\nu^*(t)<\infty, 
$$
where  $f_\nu^*$ is the decreasing rearrangement of $f$ with respect to $\nu$, defined by 
$$
f_\nu^*(t):=\inf\{y>0:\lambda_f^\nu(y)\leq t\}, \qquad \lambda_f^\nu(t):=\nu( \{x\in X : |f(x)|>t\}).
$$
If $p\geq 1$, then $L^{p,1}(\nu) \hookrightarrow L^{p}(\nu) \hookrightarrow L^{p,\infty}(\nu)$. Given a $\sigma$-finite measure space $(X,\nu)$, and parameters $0<r<p<\infty$, the quantity
$$
\vertiii{f}_{L^{p,\infty}(\nu)}:=\sup_{0<\nu(E)<\infty} \nu(E)^{\frac{1}{p}-\frac{1}{r}}\left(\int_E |f|^r d \nu \right)^{1/r}
$$
satisfies that
\begin{equation*}
    \|f\|_{L^{p,\infty}(\nu)} \leq \vertiii{f}_{L^{p,\infty}(\nu)} \leq \left(\frac{p}{p-r} \right)^{1/r} \|f\|_{L^{p,\infty}(\nu)}.
\end{equation*}
This result is classical (see \cite[Exercise 1.1.12]{grafclas}), and throughout this paper, we will refer to these inequalities as Kolmogorov's inequalities.

Given $f\in L^1_{loc}(\mathbb R^n)$, the Hardy-Littlewood maximal operator $M$ is defined by
$$
Mf(x):=\sup_{Q\ni x}\frac{1}{|Q|}\int_Q |f(y)|dy, \quad x\in \mathbb R ^n,
$$
where the supremum is taken over all cubes $Q\subseteq\mathbb R^n$ containing $x$. Muckenhoupt studied the boundedness of $M$ on Lebesgue spaces $L^p(w)$ (see \cite{mucken}). Given a positive and locally integrable function $w$, called weight, and $1<p<\infty$, 
$$
M:L^p(w)\longrightarrow L^p(w)
$$
if, and only if $w\in A_p$; that is, if 
$$
[ w]_{A_p}:= \sup_{Q} \left(\fint_Q w \right) \left( \fint_Q w  ^{1-p'} \right)^{p-1} <\infty,
$$
where we use the notation $\fint_Q w =\frac{1}{|Q|}\int_Q w(x)dx$. Moreover, if $1\le p<\infty$, 
$$
M:L^p(w)\longrightarrow L^{p, \infty}(w)
$$
if, and only if $w\in A_p$, where a weight $w\in A_1$ if 
$$
[w]_{A_1}:=\sup_Q \left( \fint_Q w\right) \Vert \chi_Q w^{-1}\Vert_{L^{\infty}(w)}= \sup_Q \left( \fint_Q w\right) (\essinf_{x\in Q} w(x))^{-1}<\infty.
$$
Buckley proved in \cite{buckley} that for $1\leq p <\infty$,
\begin{equation*}
    \Vert M\Vert_{L^{p}(w)\to L^{p,\infty}(w)} \lesssim
    [w]_{A_p}^{1/p},
\end{equation*}
and if $p>1$, then
\begin{equation*}
   \Vert M\Vert_{L^{p}(w)\to L^{p}(w)} \lesssim
   [w]_{A_p}^{\frac{1}{p-1}}.
\end{equation*}

The restricted weak type bounds of $M$ were studied in \cite{CHK,kt}. For $1\leq p < \infty$, 
$$
M:L^{p,1}(w) \longrightarrow L^{p, \infty}(w) 
$$
if, and only if $w\in A_p^{\mathcal R}$, where a weight $w$ is in $A_p^{\mathcal R}$ (also denoted by $A_{p,1}$) if 
\begin{equation*}
[w]_{A_p^{\mathcal R}} :=\sup_{Q} w(Q)^{1/p}\frac{\Vert \chi_Q w^{-1}\Vert_{L^{p',\infty}(w)}}{|Q|} <\infty,
\end{equation*}
or equivalently, if
\begin{equation*}
\Vert w \Vert _{A_p^{\mathcal R}} :=\sup_Q \sup_{E\subseteq Q} \frac{|E|}{|Q|}\left(\frac{w(Q)}{w(E)}\right)^{1/p} <\infty.
\end{equation*}
Given a measurable set $E$, we write $w(E)=\int_{E}w(x) dx$. If $w= 1$, we simply write $|E|$. We have that $[w]_{A_p^{\mathcal R}} \leq \Vert w \Vert_{A_p^{\mathcal R}} \leq p [w]_{A_p^{\mathcal R}}$. Moreover, 
\begin{equation*}
\Vert M\Vert_{L^{p,1}(w)\to L^{p,\infty}(w)} \approx [ w]_{A_p^{\mathcal R}}.
\end{equation*}
As usual, we write $A \lesssim B$ if there exists a positive constant $C>0$, independent of $A$ and $B$, such that $A\leq C B$. If the implicit constant $C$ depends on some parameter $\alpha$, we may write $\lesssim_\alpha$ at our discretion. If $A\lesssim B$ and $B\lesssim A$, then we write $A\approx B$.

We now give the definitions of some other classes of weights that will appear later. For more information about them, see \cite{CUMP,cun,DMRO,gcrf}. Define the class of weights
$$
A_{\infty}:=\bigcup_{p\geq 1} A_p = \bigcup_{p\geq 1} A_p ^{\mathcal R}.
$$
A weight $w\in A_{\infty}$ if, and only if
$$
[w]_{A_{\infty}}:= \sup_Q \frac{1}{w(Q)} \int_Q M(w\chi_Q) < \infty.
$$
This quantity is usually referred to as the Fujii-Wilson $A_\infty$ constant (see \cite{fujii}). More generally, given a weight $u$, and $p> 1$, we say that $w\in A_p(u)$ if
$$
[w]_{A_p(u)} := \sup_{Q} \left(\frac{1}{u(Q)}\int_Q wu \right) \left( \frac{1}{u(Q)}\int_Q w  ^{1-p'}u \right)^{p-1} <\infty, 
$$
and $w\in A_1(u)$ if 
$$
[w]_{A_1(u)}:=\sup_Q \left(\frac{1}{u(Q)}\int_Q wu \right) \Vert \chi_Q w^{-1}\Vert_{L^{\infty}(wu)} = \sup_Q \left(\frac{1}{u(Q)}\int_Q wu \right) (\essinf_{x\in Q} w(x))^{-1}<\infty,
$$
and as before, we define
$$
A_{\infty}(u):= \bigcup_{p\geq 1} A_p(u).
$$
If $u$ is a doubling weight for cubes in $\mathbb R^n$, and $w\in A_{\infty}(u)$, then
$$
[w]_{A_{\infty}(u)}:= \sup_Q \frac{1}{wu(Q)} \int_Q M_u(w\chi_Q) u < \infty,
$$
where
$$
M_u f(x):=\sup_{Q\ni x}\frac{1}{u(Q)}\int_Q |f(y)| u(y) dy
$$
is the weighted Hardy-Littlewood maximal operator. If $p>1$, then $M_u$ is bounded on $L^p(wu)$ if, and only if $w\in A_p(u)$, provided that $u$ is doubling. Given $s>1$, we say that a weight $w\in RH_s$ if
$$
[w]_{RH_s}:=\sup_Q \frac{|Q|}{w(Q)} \left( \fint_Q w^s\right)^{1/s}<\infty,
$$
and $w\in RH_{\infty}$ if
$$
[w]_{RH_{\infty}}:= \sup_Q \frac{|Q|}{w(Q)} \Vert \chi_Q w \Vert_{L^{\infty}(\mathbb R^n)}= \sup_Q \frac{|Q|}{w(Q)} \esssup_{x\in Q} w(x) <\infty.
$$
We have that
$$
A_\infty = \bigcup_{1<s\leq \infty } RH_s.
$$

In \cite{LOPTT}, the following multi-variable extension of the Hardy-Littlewood maximal operator was introduced in connection with the theory of multi-linear Calder\'on-Zygmund operators:
$$
\mathcal M (\vec f):= \sup_Q \prod_{i=1}^m \left( \fint_Q |f_i| \right) \chi_Q,
$$
for $\vec{f}=(f_1,\dots,f_m)$, with $f_i\in L^1_{loc}(\mathbb R ^n)$, $i=1,\dots,m$. Commonly, this operator is referred to as the \textit{curly} operator.
 For $1 \leq p_1,\dots, p_m <\infty $, $\vec P= (p_1,\dots,p_m)$, $\frac{1}{p}=\frac{1}{p_1}+ \dots + \frac{1}{p_m}$, and weights $w_1, \dots,w_m$, with $\vec w= (w_1,\dots,w_m)$, and $\nu_{\vec w}:=w_1^{p/p_1} \dots w_m ^{p/p_m}$, $$
\mathcal M : L^{p_1}(w_1) \times \dots \times L^{p_m}(w_m) \longrightarrow L^{p,\infty}(\nu_{\vec w})
$$
if, and only if $\vec w \in A_{\vec P}$; that is, if
$$
[\vec w]_{A_{\vec P}}:= \sup_Q \left( \fint_Q \nu_{\vec w}\right)^{1/p} \prod_{i=1}^m \left( \fint_Q w_i ^{1-p_i'}\right)^{1/p_i'} < \infty,
$$
where $\left( \fint_Q w_i ^{1-p_i'}\right)^{1/p_i'}$ is replaced by $(\essinf_{x\in Q} w_i (x))^{-1}$ if $p_i = 1$. Moreover, if $1 < p_1,\dots, p_m <\infty $, then 
$$
\mathcal M : L^{p_1}(w_1) \times \dots \times L^{p_m}(w_m) \longrightarrow L^{p}(\nu_{\vec w})
$$
if, and only if $\vec w \in A_{\vec P}$.

We are using throughout the paper the standard notation $
T : X_1  \times \dots \times X_m \longrightarrow X_0
$ to denote  that $T$ is a bounded operator from $ X_1  \times \dots \times X_m$ to  $ X_0$, where $X_i$ is an appropriate function space.

\subsection{Dyadic grids and sparse collections of cubes}

A general dyadic grid $\mathscr D$ is a collection of cubes in $\mathbb R^n$ with the following properties:

\begin{itemize}
\item [($a$)] For any $Q \in \mathscr D$, its side length $l_Q$ is of the form $2^k$, for some $k\in \mathbb Z$.
\item [($b$)] For all $Q,R \in \mathscr D$, $Q\cap R \in \{\emptyset, Q, R\}$.
\item [($c$)] The cubes of a fixed side length $2^k$ form a partition of $\mathbb R^n$.
\end{itemize}
The standard dyadic grid in $\mathbb R^n$ consists of the cubes $2^{-k}([0,1)^n+j)$, with $k\in \mathbb Z$ and $j\in \mathbb Z^n$. It is well-known (see \cite{hype}) that if one considers the perturbed dyadic grids 
$$\mathscr D_ \alpha :=\{2^{-k}([0,1)^n+j+\alpha):k\in \mathbb Z,j\in \mathbb Z^n\},$$
with $\alpha \in \{0,\frac{1}{3}\}^n$, then for any cube $Q\subseteq \mathbb R^n$, there exist $\alpha$ and a cube $Q_\alpha \in \mathscr D_ \alpha$ such that $Q\subseteq Q_\alpha$ and $l_{Q_\alpha}\leq 6 l_Q$. 

A collection of cubes $\mathcal S$ is said to be $\eta$-sparse if there exists $0<\eta <1$ such that for every cube $Q \in \mathcal S$, there exists a set $E_Q \subseteq Q$ with $\eta |Q| \leq |E_Q|$, and for every $Q\neq R \in \mathcal S$, $E_R \cap E_Q = \emptyset$.  

For more information about these topics, see \cite{lena}.

\subsection{Calder\'on-Zygmund operators}

We say that a function $\omega:[0,\infty) \rightarrow [0,\infty)$ is a modulus of continuity if it is continuous, increasing, sub-additive and such that $\omega(0)=0$. We say that $\omega$ satisfies the Dini condition if
$$
\Vert \omega \Vert_{\text{Dini}}:=\int_0 ^1 \frac{\omega (t)}{t}dt < \infty.
$$

We give the definition of the multi-linear $\omega$-Calder\'on-Zygmund operators. We denote by $\mathscr S(\mathbb R^n)$ the space of all Schwartz functions on $\mathbb R^n$ and by $\mathscr S'(\mathbb R^n)$ its dual space, the set of all tempered distributions on $\mathbb R^n$. 

\begin{definition}
An $m$-linear $\omega$-Calder\'on-Zygmund operator is an $m$-linear and continuous operator $T:\mathscr S (\mathbb R^n)\times \dots \times \mathscr S (\mathbb R^n)\longrightarrow \mathscr S'(\mathbb R^n)$ for which there exists a locally integrable function $K(y_0,y_1,\dots,y_m)$, defined away from the diagonal $y_0=y_1=\dots=y_m$ in $(\mathbb R^n)^{m+1}$, satisfying, for some constant $C_K>0$, the size estimate
$$
|K(y_0,y_1,\dots,y_m)| \leq \frac{C_K}{(|y_0-y_1|+\dots + |y_0-y_m|)^{nm}},
$$
for all $(y_0,y_1,\dots,y_m) \in (\mathbb R^n)^{m+1}$ with $y_0\neq y_j$ for some $j\in \{1,\dots,m\}$, and the smoothness estimate 
\begin{align*}
|K(y_0,y_1,\dots,y_i,\dots,y_m)&-K(y_0,y_1,\dots,y_i ',\dots, y_m)| \\ & \leq \frac{C_K}{(|y_0-y_1|+\dots + |y_0-y_m|)^{nm}}\omega \left( \frac{|y_i-y_i'|}{(|y_0-y_1|+\dots + |y_0-y_m|)^{nm}}\right),
\end{align*}
for $i=0,\dots,m$ and whenever $|y_i-y_i'|\leq \frac{1}{2}\max_{0\leq j \leq m} |y_i-y_j|$, and such that
$$
T(f_1,\dots,f_m)(x)=\int_{\mathbb R^n} \dots \int_{\mathbb R^n} K(x,y_1,\dots,y_m)f_1(y_1)\dots f_m(y_m)dy_1 \dots dy_m,
$$
whenever $f_1,\dots,f_m \in \mathscr C^\infty _c (\mathbb R^n)$ and $x\not \in \bigcap _{j=1}^m \text{supp } f_j$, and for some $1\leq q_1,\dots,q_m <\infty$, $T$ extends to a bounded $m$-linear operator from $L^{q_1}(\mathbb R^n)\times \dots \times L^{q_m}(\mathbb R^n)$ to $L^q(\mathbb R^n)$, with $\frac{1}{q}=\frac{1}{q_1} + \dots + \frac{1}{q_m}$.
\end{definition}

If we take $\omega(t)=t^{\varepsilon}$ for some $\varepsilon >0$, we recover the classical multi-linear Calder\'on-Zygmund operators. In general, an $m$-linear $\omega$-Calder\'on-Zygmund operator with $\omega$ satisfying the Dini condition can be extended to a bounded operator from $L^1(\mathbb R^n) \times \dots \times L^1(\mathbb R^n)$ to $L^{1/m,\infty}(\mathbb R^n)$. The multi-linear Calder\'on-Zygmund theory has been investigated by many authors. For more information on this matter, see \cite{graftor,LOPTT, lu} and the publications cited there. 

\section{Sawyer-type inequalities for maximal operators}

We devote this section to the study of a novel restricted weak type inequality that extends the classical Sawyer-type inequality (\ref{eqclassic}) for the Hardy-Littlewood maximal operator. To this end, we will need some previous results.

The following lemma contains well-known results on weights (see \cite{CUMP,cun,gcrf,kob}), but we will give most of their proofs since we need to keep track of the constants of the weights involved. 

\begin{lemma}\label{pesos}
Let $u$ and $w$ be weights.
\begin{enumerate}
\item[$(a)$] If $u\in A_1$, then $u^{-1}\in RH_{\infty}$, and $[u^{-1}]_{RH_{\infty}}\leq [u]_{A_1}$.
\item[$(b)$] If $u \in RH_{\infty}$, and $q>0$, then $u^{q}\in RH_{\infty}$. If $q\geq 1$, then $[u^{q}]_{RH_{\infty}}\leq [u]_{RH_\infty}^q$.
\item[$(c)$] If $u\in RH_{\infty}$, and $[u]_{RH_{\infty}} \leq \beta$, then there exists $r>1$, depending only on $n,\beta$, such that $u\in A_r$ and $[u]_{A_r}\leq c_{n,\beta}$. In particular, $RH_{\infty} \subseteq A_{\infty}$.
\item[$(d)$] If $u\in A_{\infty}$, and $w\in RH_{\infty}$, then $uw\in A_\infty$. 
\item[$(e)$] If $u\in A_{1} \cap RH_{\infty}$, then $u\approx 1$. 
\end{enumerate}

Fix $p\geq 1$, and $f_1,\dots,f_m\in L^1_{loc}(\mathbb R^n)$, and let $v=\prod_{i=1}^m (Mf_i)^{-1}$. 
\begin{enumerate}
\item[$(f)$] $v^p \in RH_{\infty}$, and $1\leq[v^p]_{RH_{\infty}} \leq c_{m,n,p}$.

\item[$(g)$] If $u\in A_{\infty}$, then $uv^p \in A_{\infty}$, with constant independent of $\vec f=(f_1,\dots,f_m)$.
\end{enumerate}

\end{lemma}

\begin{proof}
To prove (a), fix a cube $Q\subseteq \mathbb R^n$. By H\"older's inequality, we have that
$$
|Q|=\int_Q u^{-1/2}u^{1/2} \leq \left( \int_Q u^{-1}\right)^{1/2} \left( \int_Q u \right)^{1/2},
$$
and hence,
$$
\esssup_{x \in Q} u(x)^{-1} =(\essinf_{x\in Q}u(x))^{-1} \leq [u]_{A_1} \frac{|Q|}{u(Q)}\leq [u]_{A_1} \fint_Q u^{-1},
$$
and the desired result follows taking the supremum over all cubes $Q$.

The property (b) follows from \cite[Theorem 4.2]{cun}. Let $q\geq 1$, and fix a cube $Q\subseteq \mathbb R^n$. Then,
$$
\esssup_{x\in Q} u(x) \leq [u]_{RH_\infty} \fint_Q u \leq [u]_{RH_\infty} \left(\fint_Q u^q \right)^{1/q},
$$
from which the desired result follows, as before. 

To prove (c), fix a cube $Q\subseteq \mathbb R^n$, and a measurable set $E\subseteq Q$. Then,
$$
\frac{u(E)}{u(Q)}= \frac{1}{u(Q)}\int_Q \chi_E u \leq \frac{|E|}{u(Q)} \esssup_{x\in Q} u(x) \leq [u]_{RH_{\infty}} \frac{|E|}{|Q|} \leq \beta \frac{|E|}{|Q|}.
$$
In particular, for every $\varepsilon>0$, and $\delta:= \frac{\varepsilon}{\beta}$, if $|E|<\delta |Q|$, then $u(E)<\varepsilon u(Q)$, and the desired result follows from this fact applying the last theorem in \cite{muck}.

To prove (d), take $q,r>1$ such that $u\in A_q$ and $w\in A_r$. We will show that $uw\in A_s$, for $s:=q+r-1$. Fix a cube $Q\subseteq \mathbb R^n$. Then,
$$
\fint_Q uw \leq [w]_{RH_{\infty}} \left( \fint_Q u \right) \left( \fint_Q w\right),
$$
and in virtue of H\"older's inequality with exponent $\alpha:=1+\frac{r-1}{q-1}$,
\begin{align*}
\left(\fint _Q (uw)^{1-s'} \right)^{s-1} & \leq \left(\fint _Q u^{(1-s')\alpha} \right)^{(s-1)/\alpha} \left(\fint _Q w^{(1-s')\alpha '} \right)^{(s-1)/\alpha '} \\ & = \left(\fint _Q u^{1-q'} \right)^{q-1} \left(\fint _Q w^{1-r'} \right)^{r-1},
\end{align*}
so $[uw]_{A_s} \leq [w]_{RH_{\infty}}[u]_{A_q}[w]_{A_r}<\infty$.

The property (e) follows immediately from Corollary 4.6 in \cite{cun}.

To prove (f), observe that in virtue of \cite[Theorem 7.2.7]{grafclas}, we have that for $0< \delta <1$, $(Mf_i)^{\delta}\in A_1$, and $[(Mf_i)^{\delta}]_{A_1}\leq \frac{c_n}{1-\delta}$, $i=1,\dots,m$. In particular, $w:=\prod_{i=1}^m (Mf_i)^{\delta/m}\in A_1$, and $[w]_{A_1}\leq \prod_{i=1}^m [(Mf_i)^{\delta}]_{A_1}^{1/m}\leq \frac{c_n}{1-\delta}$. Since $v^p=w^{-mp/\delta}$, it follows from (a) and (b) that
$$
[v^p]_{RH_\infty} \leq [w^{-1}]^{mp/\delta}_{RH_\infty}  \leq [w]^{mp/\delta}_{A_1} \leq \left(\frac{c_n}{1-\delta}\right)^{mp/\delta},
$$
so
$$
1\leq [v^p]_{RH_\infty} \leq c_{m,n,p}:=\inf_{0< \delta < 1} \left(\frac{c_n}{1-\delta}\right)^{mp/\delta}.
$$

To prove (g), we already know by (f) that $v^p\in RH_{\infty}$, with constant bounded by $c_{m,n,p}$, so by (c), there exists $r>1$, depending only on $m,n,p$, such that $[v^p]_{A_r}\leq C_{m,n,p}$. By (d), for $q>1$ such that $u\in A_q$, and $s=s(m,n,p,q)=q+r-1$, $[uv^p]_{A_s}\leq \widetilde{C}_{m,n,p}[u]_{A_q}<\infty$.  
\end{proof}

The next lemma gives a result on weights that will be handy later on.

\begin{lemma}\label{weights}
Let $u$ and $v$ be weights, and suppose that $u\in A_{\infty}$. Then, $uv\in A_{\infty}$ if, and only if $v\in A_{\infty}(u)$.
\end{lemma}

\begin{proof}
Let us first assume that $uv\in A_{\infty}$. Since $u\in A_{\infty}$, there exists $s>1$ such that $u\in RH_s$, and since $uv\in A_{\infty}$, there exists $r>1$ such that $uv\in A_r$. Take $q:=\frac{rs}{s-1}>1$. We will show that $v\in A_q(u)$. Fix a cube $Q$. Then,
\begin{align*}
I_Q:=\left(\frac{1}{u(Q)} \int _Q v u \right) \left(\frac{1}{u(Q)} \int _Q v^{1-q'} u \right)^{q-1} = \left( \frac{|Q|}{u(Q)} \right)^q \left(\frac{1}{|Q|} \int _Q v u \right) \left(\frac{1}{|Q|} \int _Q (vu)^{1-q'} u^{q'} \right)^{q-1}.
\end{align*}
Take $\alpha:=\frac{q-1}{r-1}=1+\frac{r}{(r-1)(s-1)}>1$ and observe that $(1-q')\alpha= 1-r'$, $\frac{q-1}{\alpha}=r-1$, $q' \alpha ' = s$, and $\frac{q-1}{\alpha'}=\frac{q}{s}$. Using H\"older's inequality with exponent $\alpha$, we get that
\begin{align*}
\left(\frac{1}{|Q|} \int _Q (vu)^{1-q'} u^{q'} \right)^{q-1} & \leq \left(\frac{1}{|Q|} \int _Q (vu)^{(1-q')\alpha} \right)^{(q-1)/\alpha} \left(\frac{1}{|Q|} \int _Q u^{q'\alpha '} \right)^{(q-1)/\alpha '} \\ & = \left(\frac{1}{|Q|} \int _Q (vu)^{1-r'} \right)^{r-1} \left(\frac{1}{|Q|} \int _Q u^{s} \right)^{q/s} \\ & \leq [u]_{RH_s}^q \left(\frac{1}{|Q|} \int _Q (vu)^{1-r'} \right)^{r-1} \left(\frac{u(Q)}{|Q|} \right)^{q}.
\end{align*}
Hence,
$$
I_Q\leq [u]_{RH_s}^q \left(\frac{1}{|Q|} \int _Q v u \right) \left(\frac{1}{|Q|} \int _Q (vu)^{1-r'} \right)^{r-1} \leq [u]_{RH_s}^q [uv]_{A_r},
$$
and $[v]_{A_q(u)}=\sup_Q I_Q \leq [u]_{RH_s}^q [uv]_{A_r} < \infty$.

For the converse, let us assume that $v\in A_{\infty}(u)$. It follows from Theorem 3.1 in \cite{DMRO} that there exist $\delta, C>0$ such that for every cube $Q\subseteq \mathbb R^n$ and every measurable set $E\subseteq Q$,
$$
\frac{u(E)}{u(Q)} \leq C \left(\frac{uv(E)}{uv(Q)} \right)^\delta.
$$
Similarly, since $u\in A_\infty$, there exist $\varepsilon,c >0$ such that for every cube $Q\subseteq \mathbb R^n$ and every measurable set $E\subseteq Q$,
$$
\frac{|E|}{|Q|} \leq c \left(\frac{u(E)}{u(Q)} \right)^\varepsilon,
$$
so for every cube $Q\subseteq \mathbb R^n$ and every measurable set $E\subseteq Q$,
$$
\frac{|E|}{|Q|} \leq cC^\varepsilon \left(\frac{uv(E)}{uv(Q)} \right)^{\varepsilon \delta},
$$
and hence, $uv\in A_\infty$.
\end{proof}

\begin{remark}
This result is an extension of Lemma 2.1 in \cite{CUMP}, where it is shown that if $u\in A_1$ and $v\in A_{\infty}(u)$, then $uv\in A_{\infty}$. 
\end{remark}

We introduce a weighted version of the dyadic Hardy-Littlewood maximal operator.

\begin{definition}
Let $\mathscr D$ be a general dyadic grid in $\mathbb R^n$, and let $u$ be a weight. For a measurable function $f$, we consider the function 
$$
M_u^{\mathscr D} f(x):=\sup_{\mathscr D \ni Q \ni x}  \frac{1}{u(Q)} \int _Q |f(y)|u(y)dy, \quad x \in \mathbb R^n,
$$
where the supremum is taken over all cubes $Q\in \mathscr D$ that contain $x$. If $u = 1$, we simply write $M^{\mathscr D} f$.
\end{definition}

The following bound for the operator $M_u^{\mathscr D}$ is essential.

\begin{theorem}\label{dyadic}
Let $\mathscr D$ be a general dyadic grid in $\mathbb R^n$, and let $u$ and $v$ be weights. If $u\in A_\infty$ and $uv\in A_\infty$, then there exists a constant $C_{u,v}$, independent of $\mathscr D$, such that for every measurable function $f$,
$$
\left \Vert \frac{M_u ^{\mathscr D} (fv)}{v}\right \Vert_{L^{1,\infty}(uv)}
\leq C_{u,v}\, \int_{\mathbb R^n} |f(x)| u(x) v(x) dx.
$$
\end{theorem}

\begin{proof}
In virtue of Lemma~\ref{weights}, $v\in A_{\infty}(u)$ and hence, this theorem follows from the proof of Theorem 1.4 in \cite{CUMP}.
\end{proof}

\begin{remark}\label{constant}
If we examine the proof of Theorem 1.4 in \cite{CUMP}, and we combine it with Appendix A in \cite{umpbook}, we can take
$$
C_{u,v}=2^q(2^n r [uv]_{A_r^{\mathcal R}})^{r(q-1)}\left \Vert M_u \right \Vert_{L^q(uv^{1-q})}^q,
$$
where $r,q > 1$ are such that $uv\in A_r^{\mathcal R}$ and $v\in A_{q'}(u)$.
\end{remark}

\begin{remark}
The bound of Theorem~\ref{dyadic} also holds for the weighted Hardy-Littlewood maximal operator $M_u$, with constant 
$$C:=2^n 6^{np}p^p[u]_{A_p^{\mathcal R}}^p C_{u,v},$$
where $p\geq 1$ is such that $u\in A_p^{\mathcal R}$.
\end{remark}

We can now state and prove the main result of this section.

\begin{theorem}\label{sawyer}
Fix $p\geq 1$, and let $u$ and $v$ be weights such that $u\in A_{p}^{\mathcal R}$ and $uv^p\in A_{\infty}$. Then, there exists a constant $C>0$ such that for every measurable function $f$, 
\begin{equation*}
    \left \Vert \frac{Mf}{v}\right \Vert_{L^{p,\infty}(uv^p)} \leq C \Vert f \Vert_{L^{p,1}(u)}.
\end{equation*}
\end{theorem}

\begin{proof}
It is known (see \cite{hype,lerner}) that there exists a collection $\{\mathscr D _ \alpha \}_{\alpha}$ of $2^n$ general dyadic grids in $\mathbb R^n$ such that
$$
Mf \leq 6^n \sum_{\alpha = 1} ^{2^n} M^{\mathscr D _ \alpha} f.
$$
Hence,
$$
\left \Vert \frac{Mf}{v}\right \Vert_{L^{p,\infty}(uv^p)} \leq 12^n \sum_{\alpha = 1}^{2^n} \left \Vert \frac{M^{\mathscr D _ \alpha} f}{v}\right \Vert_{L^{p,\infty}(uv^p)},
$$
and it suffices to establish the result for the operator $M^{\mathscr D}$, with $\mathscr D$ a general dyadic grid in $\mathbb R^n$. 

We first discuss the case $p=1$, which was proved in \cite{CUMP}. We reproduce the proof here keeping track of the constants. By the definition of the $A_1$ condition, 
$$
\frac{1}{|Q|} \int_Q |f| \leq [u]_{A_1}\frac{1}{u(Q)} \int _Q |f|u,
$$
for every cube $Q\in \mathscr D$, so we get that $M^{\mathscr D} f \leq [u]_{A_1} M^{\mathscr D}_u f$. This estimate combined with Theorem~\ref{dyadic} gives that
$$
\left \Vert \frac{M^{\mathscr D } f}{v}\right \Vert_{L^{1,\infty}(uv)} \leq [u]_{A_1} \left \Vert \frac{M^{\mathscr D }_u (vf/v)}{v}\right \Vert_{L^{1,\infty}(uv)}\leq [u]_{A_1} C_{u,v} \int_{\mathbb R^n} |f|u,
$$
and hence, the desired result follows, with $C=24^n [u]_{A_1} C_{u,v}$.

Now, we discuss the case $p>1$. Let us take $f=\chi_E$, with $E$ a measurable set in $\mathbb R^n$, and fix a cube $Q\in \mathscr D$. As before, by the definition of the $A_p ^{\mathcal R}$ condition, 
$$
\frac{1}{|Q|} \int_Q f \leq \Vert u \Vert_{A_p ^{\mathcal R}} \left( \frac{u(E\cap Q)}{u(Q)}\right)^{1/p},
$$
so we get that $M^{\mathscr D} (\chi_E) \leq p[u]_{A_p^{\mathcal R}} (M^{\mathscr D}_u (\chi_E))^{1/p}$. In particular,
$$
\left \Vert \frac{M^{\mathscr D } (\chi_E)}{v}\right \Vert_{L^{p,\infty}(uv^p)} \leq p[u]_{A_p^{\mathcal R}} \left \Vert \frac{M^{\mathscr D }_u (\chi_E)}{v^p}\right \Vert_{L^{1,\infty}(uv^p)}^{1/p}. 
$$
We can now apply Theorem~\ref{dyadic} to conclude that
$$
\left \Vert \frac{M^{\mathscr D }_u (\chi_E)}{v^p}\right \Vert_{L^{1,\infty}(uv^p)}=\left \Vert \frac{M^{\mathscr D }_u (v^p\chi_E/v^p)}{v^p}\right \Vert_{L^{1,\infty}(uv^p)} \leq C_{u,v^p} u(E). 
$$

Combining all the previous estimates, we have that
$$
\left \Vert \frac{M(\chi_E)}{v}\right \Vert_{L^{p,\infty}(uv^p)} \leq 24^n [u]_{A_p ^{\mathcal R}} C_{u,v^p} ^{1/p} \Vert \chi_E \Vert _{L^{p,1}(u)}.
$$

Since $p>1$, $L^{p,\infty}(uv^p)$ is a Banach space, and by standard arguments (see \cite[Exercise 1.4.7]{grafclas}), we can extend the previous estimate to arbitrary measurable functions $f$, gaining a factor of $4p'$ in the constant. Hence, the desired result follows, with $C=4\cdot 24^n p' [u]_{A_p^{\mathcal R}} C_{u,v^p}^{1/p}$.
\end{proof}

\begin{remark}\label{openq}
For $p=1$ and $u\in A_1$, a more general version of Theorem~\ref{sawyer} was established in \cite{kop}, replacing the hypothesis that $uv\in A_\infty$ by the weaker assumption that $v\in A_ \infty$. It is unknown to us whether the hypothesis that $uv^p\in A_\infty$ can be replaced by $v\in A_\infty$ when $p>1$. 
\end{remark}

In virtue of Lemma \ref{pesos}, if $u\in A_\infty$ and $v \in RH_\infty$, then for every $p\geq 1$, $uv^p\in A_\infty$, and we have a whole class of non-trivial examples of weights that satisfy the hypotheses of Theorem~\ref{sawyer}.

Observe that the conclusion of Theorem~\ref{sawyer} is completely elementary if $p>1$ and $u\in A_p$, since
\begin{align*}
\left \Vert \frac{Mf}{v} \right \Vert_{L^{p,\infty}(uv^p)} & \leq \left \Vert \frac{Mf}{v} \right \Vert_{L^{p}(uv^p)}=\left \Vert Mf \right \Vert_{L^{p}(u)} \lesssim  [u]_{A_p}^{\frac{1}{p-1}} \Vert f \Vert_{L^p(u)} \lesssim  [u]_{A_p}^{\frac{1}{p-1}} \Vert f \Vert_{L^{p,1}(u)}.
\end{align*}
However, this argument doesn't work in the general case, because the inequality
$$
\left \Vert \frac{h}{v}\right \Vert_{L^{p,\infty}(uv^p)} \lesssim \left \Vert h\right \Vert_{L^{p,\infty}(u)}
$$
may fail for some measurable functions $h$ on $\mathbb R^n$, and arbitrary weights $u$ and $v$, as can be seen by choosing $h(x)= |x|^{-\frac{n}{p}}\chi_{\{y\in \mathbb R^n \, : \, |y|\ge 1\}}(x)$, $u=1$, and $v(x) = h(x)+\chi_{\{y\in \mathbb R^n \, : \, |y| < 1\}}(x)$, with $0<p<\infty$.

To provide applications of Theorem~\ref{sawyer} we need to give a more precise estimate of the constant $C$ that appears there in terms of the corresponding constants of the weights involved. We achieve this in the following lemma.

\begin{lemma}\label{creciente}
In Theorem~\ref{sawyer}, if $r\geq 1$ is such that $uv^p \in A_r^{\mathcal R}$, then one can take
$$
C=\mathscr E_{r,p}^n([u]_{A_p^{\mathcal R}},[uv^p]_{A_r^{\mathcal R}}),
$$
where $\mathscr E_{r,p}^n:[1,\infty)^2 \longrightarrow \mathbb (0,\infty)$ is a function that increases in each variable, and it depends only on $r$, $p$, and the dimension $n$.
\end{lemma}

\begin{proof}
We first discuss the case when $r>1$. We already know that we can take
$$
C=\left \{\begin{array}{lr}
24^n [u]_{A_1} C_{u,v}, & p=1, \\
4\cdot 24^n p' [u]_{A_p^{\mathcal R}} C_{u,v^p}^{1/p}, & p>1,
\end{array}\right.
$$
and in virtue of Remark~\ref{constant},
$$
C_{u,v^p}=2^q(2^n r [uv^p]_{A_r^{\mathcal R}})^{r(q-1)}\left \Vert M_u \right \Vert_{L^q(uv^{p(1-q)})}^q,
$$
where $r,q > 1$ are such that $uv^p\in A_r^{\mathcal R}$ and $v^p\in A_{q'}(u)$. For convenience, we write $V:=v^p$. Let us first bound the factor $\left \Vert M_u \right \Vert_{L^q(uV^{1-q})}^q$. For the space of homogeneous type $(\mathbb R^n,d_\infty,u(x)dx)$, it follows from the proof of Theorem 1.3 in \cite{hpr} that
$$
\left \Vert M_u \right \Vert_{L^q(uV^{1-q})}^q \leq 2^{q-1}q'40^{qD_u}(1+6\cdot 800^{D_u})[V]_{A_\infty(u)}[V]^{q-1}_{A_{q'}(u)},
$$
where $D_u:=p \log_2 (2^n p [u]_{A_p^{\mathcal R}})$. Now, given a cube $Q\subseteq \mathbb R^n$, and applying H\"older's inequality with exponent $q$, we have that
\begin{align*}
\int_Q M_u(V\chi_Q) u & =\int_Q \frac{M_u(V\chi_Q)}{V} uV \leq \left \Vert \frac{M_u(V\chi_Q)}{V}\right \Vert_{L^q(uV)}uV(Q)^{1/q'} \\ & = \left \Vert M_u(V\chi_Q) \right \Vert_{L^q(uV^{1-q})}uV(Q)^{1/q'} \\ & \leq \left \Vert M_u \right \Vert_{L^q(uV^{1-q})} \left \Vert V\chi_Q \right \Vert_{L^q(uV^{1-q})}uV(Q)^{1/q'} \\ & = \left \Vert M_u \right \Vert_{L^q(uV^{1-q})} uV(Q),
\end{align*}
and taking the supremum over all cubes $Q$, we get that $[V]_{A_\infty(u)}\leq \left \Vert M_u \right \Vert_{L^q(uV^{1-q})} $. Combining the previous estimates, we obtain that
$$
\left \Vert M_u \right \Vert_{L^q(uV^{1-q})}^q \leq (2^{q-1}q'40^{qD_u}(1+6\cdot 800^{D_u}))^{q'}[V]^{q}_{A_{q'}(u)}.
$$
Now, we will bound the factor $[V]^{q}_{A_{q'}(u)}$. In virtue of \cite[Proposition 2.2]{hype}, and using the definitions of $[u]_{A_{2p}}$ and $[u]_{A_{p}^{\mathcal R}}$, and Kolmogorov's inequalities, we can deduce that
$$
[u]_{A_\infty} \leq c_n [u]_{A_{2p}} \leq (2p-1)^{2p-1} c_n [u]_{A_p^{\mathcal R}}^{2p} =:c_{p,n}[u]_{A_p^{\mathcal R}}^{2p}, 
$$
and applying Theorem 2.3 in \cite{hpr}, $u\in RH_s$ for $s=1+\frac{1}{2^{n+1}c_{p,n}[u]_{A_p^{\mathcal R}}^{2p}-1}$, and $[u]_{RH_s}\leq 2$. Since $uV\in A_{2r}$, Lemma~\ref{weights} tells us that if we choose $q'=2rs'$, then
$$
[V]_{A_{q'}(u)}^q \leq [u]_{RH_s}^{qq'}[uV]_{A_{2r}}^q \leq 2^{qq'}(2r-1)^{q(2r-1)}[uV]_{A_r^{\mathcal R}}^{2rq}.
$$
Finally, observe that $q'=2^{n+2}r c_{p,n}[u]_{A_p^{\mathcal R}}^{2p}$, and $1<q\leq2$, so 
\begin{align*}
C_{u,V} & \leq 2^2(2^n r [uV]_{A_r^{\mathcal R}})^{r} \times (2q'40^{2D_u}(1+6\cdot 800^{D_u}))^{q'} \times 2^{2q'}(2r-1)^{4r-2}[uV]_{A_r^{\mathcal R}}^{4r} \\ & \leq 2^{2+nr}(2r-1)^{4r-2}r^r[uv^p]^{5r}_{A_r^{\mathcal R}} \left(2^{n+5}r c_{p,n}[u]_{A_p^{\mathcal R}}^{2p}40^{5p \log_2 (2^n p [u]_{A_p^{\mathcal R}})}\right)^{2^{n+2}r c_{p,n}[u]_{A_p^{\mathcal R}}^{2p}}\\ & =: C_{r,p}^n([u]_{A_p^{\mathcal R}},[uv^p]_{A_r^{\mathcal R}}),
\end{align*}
and the desired result follows, with 
$$
\mathscr E_{r,p}^n([u]_{A_p^{\mathcal R}},[uv^p]_{A_r^{\mathcal R}}) =\left \{\begin{array}{lr}
24^n [u]_{A_1} C_{r,1}^n([u]_{A_1},[uv]_{A_r^{\mathcal R}}), & p=1, \\
4\cdot 24^n p' [u]_{A_p^{\mathcal R}} C_{r,p}^n([u]_{A_p^{\mathcal R}},[uv^p]_{A_r^{\mathcal R}})^{1/p}, & p>1.
\end{array}\right.
$$

The case when $r=1$ follows, for example, from the case when $r=2$ and the fact that if $uv^p \in A_1$, then $[uv^p]_{A_2^{\mathcal R}}\leq [uv^p]_{A_2}^{1/2}\leq [uv^p]_{A_1}^{1/2}$.
\end{proof}

\begin{remark}
It would be interesting to obtain the sharp dependence of $C$ on the constants of the weights involved; our results are most certainly far from optimal.
\end{remark}

\section{Applications}

In this section, we will provide several applications of the Sawyer-type inequality established in Theorem \ref{sawyer}, obtaining mixed restricted weak type estimates for multi-variable maximal operators, sparse operators and Calder\'on-Zygmund operators.

The first result that we present is the converse of Theorem 3.3 in \cite{CR}, that was left as an open question. Combining both theorems, we obtain the complete characterization of the restricted weak type bounds of the operator $M^\otimes$ for $A_\infty$ weights. 

\begin{theorem}\label{prodhl}
Let $1\leq p_1,\dots, p_m <\infty $, and let $\frac{1}{p}=\frac{1}{p_1}+ \dots + \frac{1}{p_m}$. Let $w_1, \dots,w_m$ be weights, with $w_i \in A_{p_i}^{\mathcal R}$, $i=1,\dots,m$, and write $\nu_{\vec w}=w_1^{p/p_1}\dots w_m ^{p/p_m}$. Then, there exists a constant $C>0$ such that the inequality
$$
 \| M^\otimes (\vec f) \|_{L^{p,\infty}(\nu_{\vec w})} \leq C \prod_{i=1}^m \left \| f_i \right \|_{L^{p_i,1}(w_i)}
$$ 
holds for every vector of measurable functions $\vec f=(f_1,\dots,f_m)$.
\end{theorem}

\begin{proof}
The case when $p_1=\dots=p_m=1$ was proved in \cite{LOPTT}, and we build upon that proof to demonstrate the remaining cases. 

We can assume, without loss of generality, that $f_i \in L^{\infty}_c(\mathbb R^n)$, $i=1,\dots,m$. Fix $t >0$ and define 
$$E_t:=\{x\in \mathbb R^n : t<M^\otimes(\vec f)(x) \leq 2t\}.
$$
For $i=1,\dots,m$, and taking $\tilde v_i:=\prod_{j\neq i} (Mf_j)^{-1}$, we have that $$E_t = \{x\in \mathbb R^n : t\tilde v_i(x) < Mf_i(x) \leq 2t \tilde v_i(x)  \}.$$ Using the fact that $\tilde v_i\in RH_{\infty}$, with constant independent of $\vec f$ (see Lemma \ref{pesos}), H\"older's inequality, and Theorem~\ref{sawyer}, we obtain that
\begin{align*}
\lambda_{M^\otimes(\vec f)}^{\nu_{\vec w}}(t) - \lambda_{M^\otimes(\vec f)}^{\nu_{\vec w}}(2t)  &= \int_{E_t} \nu_{\vec w} \leq \int_{E_t} \left(\frac{M^\otimes(\vec f)}{t} \right)^p \nu_{\vec w} \\ & \leq \frac{1}{t^p} \prod_{i=1}^m  \left(\int_{E_t}(Mf_i)^{p_i}w_i \right)^{p/p_i} \\ & \leq 2^{mp} t^{(m-1)p} \prod_{i=1}^m \left(\int_{\left \{ \frac{Mf_i}{\tilde v_i}>t\right \}}\tilde v_i^{p_i}w_i \right)^{p/p_i} \\ & \leq 2^{mp} C_1^p \dots C_m ^p \frac{1}{t^p} \prod_{i=1}^m \left \| f_i \right \|_{L^{p_i,1}(w_i)}^p.
\end{align*} 
Iterating this result, we get that for each $t>0$ and every natural number $N$,
$$
\lambda_{M^\otimes(\vec f)}^{\nu_{\vec w}}(t) \leq 2^{mp} C_1^p \dots C_m ^p \left(\sum_{j=0}^N \frac{1}{2^{jp}} \right) \frac{1}{t^p} \prod_{i=1}^m \left \| f_i \right \|_{L^{p_i,1}(w_i)}^p + \lambda_{M^\otimes(\vec f)}^{\nu_{\vec w}}(2^{N+1}t),
$$
and letting $N$ tend to infinity, the last term vanishes, and we conclude that
$$
\lambda_{M^\otimes(\vec f)}^{\nu_{\vec w}}(t) \leq \frac{2^{(m+1)p}}{2^p -1} C_1^p \dots C_m ^p \frac{1}{t^p} \prod_{i=1}^m \left \| f_i \right \|_{L^{p_i,1}(w_i)}^p.
$$

Observe that in virtue of Lemma \ref{pesos}, for $i=1,\dots,m$, we have that $w_i\tilde v_i ^{p_i}\in A_{s_i}$, where $s_i>1$ depends only on $m,n,p_i$, and 
\begin{equation*}
    [w_i\tilde v_i ^{p_i}]_{A_{s_i}^{\mathcal R}}^{s_i} \leq [w_i\tilde v_i ^{p_i}]_{A_{s_i}}\lesssim_{m,n,p_i} [w_i]_{A_{2p_i}}\lesssim_{m,n,p_i} [w_i]_{A_{p_i}^{\mathcal R}}^{2p_i},
\end{equation*} so by Lemma \ref{creciente}, we have that $C_i \leq \mathscr E ^n _ {s_i,p_i}([w_i]_{A_{p_i}^{\mathcal R}},C_{m,n,p_i}[w_i]_{A_{p_i}^{\mathcal R}} ^{2p_i/s_i})$, and hence, the desired result follows, with
$$
C=\frac{2^{m+1}}{(2^p -1)^{1/p}} \prod_{i= 1}^m \mathscr E ^n _ {s_i,p_i}([w_i]_{A_{p_i}^{\mathcal R}},C_{m,n,p_i}[w_i]_{A_{p_i}^{\mathcal R}} ^{2p_i/s_i}),
$$ 
which depends on the constants of the weights $w_1, \dots, w_m$ in an increasing way.
\end{proof}

The next application that we provide is an extension of Theorem~\ref{sawyer} to the multi-variable setting, which in turn, extends Theorem~\ref{prodhl}. The proof is based on the previous one, and is similar to that of Theorem 1.4 in \cite{kob}.

\begin{theorem}\label{msawyer}
Let $1\leq p_1,\dots, p_m <\infty $, and let $\frac{1}{p}=\frac{1}{p_1}+ \dots + \frac{1}{p_m}$. Let $w_1, \dots,w_m$ be weights, with $w_i \in A_{p_i}^{\mathcal R}$, $i=1,\dots,m$, and write $\nu_{\vec w}=w_1^{p/p_1}\dots w_m ^{p/p_m}$. Let $v$ be a weight such that $\nu_{\vec w}v^p$ is a weight, and $w_iv^{p_i}\in A_\infty$, $i=1,\dots,m$. Then, there exists a constant $C>0$ such that the inequalities
$$
\left \| \frac{\mathcal M (\vec f)}{v} \right \|_{L^{p,\infty}(\nu_{\vec w}v^p)} \leq \left \| \frac{M^\otimes (\vec f)}{v} \right \|_{L^{p,\infty}(\nu_{\vec w}v^p)} \leq C \prod_{i=1}^m \left \| f_i \right \|_{L^{p_i,1}(w_i)}
$$ 
hold for every vector of measurable functions $\vec f=(f_1,\dots,f_m)$.
\end{theorem}

\begin{proof}
The first inequality follows from the fact that $\mathcal M(\vec f) \leq M^\otimes (\vec f)$. For the second one, we can assume, without loss of generality, that $f_i \in L^{\infty}_c(\mathbb R^n)$, $i=1,\dots,m$. Fix $y,R >0$ and define 
$$E_y^R:=\{x\in \mathbb R^n : |x|<R, \, yv(x)<M^\otimes(\vec f)(x) \leq 2yv(x)\}.
$$
For $i=1,\dots,m$, and taking $\tilde v_i:=\prod_{j\neq i} (Mf_j)^{-1}$, and $v_i:=\tilde v_i v$, we have that $$E_y^R = \{x\in \mathbb R^n : |x|<R, \, y v_i(x) < Mf_i(x) \leq 2y v_i(x) \}.$$ Since $\tilde v_i\in RH_{\infty}$, and $w_i v^{p_i} \in A_{\infty}$, we have that $w_i  v_i ^{p_i} \in A_{\infty}$, with constant independent of $\vec f$ (see Lemma \ref{pesos}). In virtue of H\"older's inequality and Theorem~\ref{sawyer}, we get that
\begin{align*}
  &\nu_{\vec w}v^p \left(\left \{x\in \mathbb R^n : |x|<R, \, \frac{M^\otimes(\vec f)(x)}{v(x)}>y\right \}\right) - \nu_{\vec w}v^p \left(\left \{x\in \mathbb R^n : |x|<R, \, \frac{M^\otimes(\vec f)(x)}{v(x)}>2y\right \}\right) \\ &= \int_{E_y^R} \nu_{\vec w}v^p \leq \int_{E_y^R} \left(\frac{M^\otimes(\vec f)}{y} \right)^p \nu_{\vec w} \leq \frac{1}{y^p} \prod_{i=1}^m  \left(\int_{E_y^R}(Mf_i)^{p_i}w_i \right)^{p/p_i} \\ & \leq 2^{mp} y^{(m-1)p} \prod_{i=1}^m \left(\int_{\left \{\frac{Mf_i}{v_i}>y\right \}}v_i^{p_i}w_i \right)^{p/p_i} \leq 2^{mp} C_1^p \dots C_m ^p \frac{1}{y^p} \prod_{i=1}^m \left \| f_i \right \|_{L^{p_i,1}(w_i)}^p.
\end{align*} 
Iterating this result, we deduce that for each $y>0$ and every natural number $N$,
\begin{align*}
    &\nu_{\vec w}v^p \left(\left \{x\in \mathbb R^n : |x|<R, \, \frac{M^\otimes(\vec f)(x)}{v(x)}>y\right \}\right) \leq 2^{mp} C_1^p \dots C_m ^p \left(\sum_{j=0}^N \frac{1}{2^{jp}} \right) \frac{1}{y^p} \prod_{i=1}^m \left \| f_i \right \|_{L^{p_i,1}(w_i)}^p \\ & + \nu_{\vec w}v^p \left(\left \{x\in \mathbb R^n : |x|<R, \, \frac{M^\otimes(\vec f)(x)}{v(x)}>2^{N+1}y\right \}\right),
\end{align*}
and letting first $N$ tend to infinity, and then $R$, the last term vanishes, and we conclude that
$$
\lambda_{\frac{M^\otimes(\vec f)}{v}}^{\nu_{\vec w}v^p}(y) \leq \frac{2^{(m+1)p}}{2^p -1} C_1^p \dots C_m ^p \frac{1}{y^p} \prod_{i=1}^m \left \| f_i \right \|_{L^{p_i,1}(w_i)}^p.
$$

For $i=1,\dots,m$, if we take $q_i>1$ such that $w_i v^{p_i}\in A_{q_i}^{\mathcal R}$, in virtue of Lemma \ref{pesos}, we have that $w_i v_i ^{p_i}\in A_{s_i}$, where $s_i>1$ depends only on $m,n,p_i,q_i$, and $[w_i v_i ^{p_i}]_{A_{s_i}^{\mathcal R}}^{s_i}\lesssim_{m,n,p_i,q_i} [w_iv^{p_i}]_{A_{q_i}^{\mathcal R}}^{2q_i}$, so by Lemma \ref{creciente}, we have that $C_i \leq \mathscr E ^n _ {s_i,p_i}([w_i]_{A_{p_i}^{\mathcal R}},C_{m,n,p_i,q_i}[w_iv^{p_i}]_{A_{q_i}^{\mathcal R}} ^{2q_i/s_i})$, and hence, the desired result follows, with
$$
C=\frac{2^{m+1}}{(2^p -1)^{1/p}} \prod_{i= 1}^m \mathscr E ^n _ {s_i,p_i}([w_i]_{A_{p_i}^{\mathcal R}},C_{m,n,p_i,q_i}[w_iv^{p_i}]_{A_{q_i}^{\mathcal R}} ^{2q_i/s_i}).
$$ 
\end{proof}

\begin{remark}
In the case when $p_1=\dots=p_m=1$, the previous result is a corollary of Theorem 1.4 in \cite{kob}.
\end{remark}

Observe that if we take weights $w_i\in A_{p_i}^{\mathcal R}$, $i=1,\dots,m$, and $v\in RH_{\infty}$, then the hypotheses of Theorem~\ref{msawyer} are satisfied.

The next result will be crucial to work with Calder\'on-Zygmund operators in the mixed restricted weak setting.

\begin{theorem}\label{sparsemax} 
Let $0<p<\infty$, let $\mathcal S$ be an $\eta$-sparse collection of cubes, and let $v,w$ be weights. Suppose that there exists $0<\varepsilon \leq 1$ such that $\varepsilon < p$, $wv^{-\varepsilon} \in A_\infty$, and 
$$
[v^{-\varepsilon}]_{RH_\infty(w)}:=\sup_Q \frac{w(Q)}{wv^{-\varepsilon}(Q)}\Vert \chi_Q v^{-\varepsilon}\Vert_{L^{\infty}(w)}<\infty. 
$$
Then, there exists a constant $C>0$, independent of $\mathcal S$, such that the inequality
$$
\left \| \frac{\mathcal{A}_{\mathcal S}(\vec{f})}{v} \right \|_{L^{p,\infty}(w)} \leq C \left \| \frac{\mathcal{M}(\vec{f})}{v} \right \|_{L^{p,\infty}(w)} 
$$
holds for every vector of measurable functions $\vec f=(f_1,\dots,f_m)$.
\end{theorem}

\begin{proof}
In virtue of Kolmogorov's inequalities, we obtain that
$$
\left \| \frac{\mathcal{A}_{\mathcal S}(\vec{f})}{v} \right \|_{L^{p,\infty}(w)} \leq \sup_{0<w(F)<\infty} \left \| \frac{\mathcal{A}_{\mathcal S}(\vec{f})}{v}\chi_F \right \|_{L^{\varepsilon}(w)}w(F)^{\frac{1}{p}-\frac{1}{\varepsilon}},
$$ 
where the supremum is taken over all measurable sets $F$ with $0<w(F)<\infty$. For one of such sets $F$, and $W:=wv^{-\varepsilon}$, we have that
$$
\begin{aligned}
&\left \| \frac{\mathcal{A}_{\mathcal S}(\vec{f})}{v}\chi_F \right \|_{L^{\varepsilon}(w)}^\varepsilon \leq \int_{\mathbb R^n} \sum_{Q \in \mathcal S} \chi_{Q} \left(\frac{\prod_{i=1}^m\fint_{Q}|f_i|}{v}\right)^\varepsilon \chi_F w \\ & = \sum_{Q \in \mathcal S}  \left(\prod_{i=1}^m \fint_{Q}|f_i|\right)^\varepsilon \left(\frac{1}{W(3Q)}\int_{Q}\chi_F W\right)W(3Q)=:I.
\end{aligned}
$$  

Since $W\in A_{\infty}$, there exists $r\geq 1$ such that $W\in A_r^{\mathcal R}$. Hence, 
$$\sup_Q \sup_{E\subseteq Q} \frac{|E|}{|Q|}\left(\frac{W(Q)}{W(E)}\right)^{1/r} = \Vert W \Vert_{A_r^{\mathcal R}} <\infty.$$
By hypothesis, $\mathcal S$ is $\eta$-sparse, so for each $Q\in \mathcal S$, $W(3Q)\leq (\frac{3^n}{\eta}\Vert W \Vert_{A_r^{\mathcal R}})^r W(E_Q)$. Using this, we get that
$$
\begin{aligned}
I & \leq \left(\frac{3^n}{\eta}\Vert W \Vert_{A_r^{\mathcal R}}\right)^r \sum_{Q\in \mathcal S} \left(\prod_{i=1}^m\fint_{Q}|f_i|\right)^\varepsilon \left(\frac{1}{W(3Q)}\int_{Q}\chi_F W\right)W(E_Q) \\ & = \left(\frac{3^n}{\eta}\Vert W \Vert_{A_r^{\mathcal R}}\right)^r \sum_{Q\in \mathcal S} \int_{E_Q} \left(\prod_{i=1}^m \fint_{Q}|f_i|\right)^\varepsilon \left(\frac{1}{W(3Q)}\int_{Q}\chi_F W\right) W =: II.
\end{aligned}
$$

The sides of an $n$-dimensional cube have Lebesgue measure $0$ in $\mathbb R^n$, so we can assume that the cubes in $\mathcal S$ are open. For $Q\in \mathcal S$ and $z\in E_Q$, we define $Q^z:=Q(z,l_Q)$, the open cube of center $z$ and side length twice the side length of $Q$. We have that $E_Q \subseteq Q \subseteq Q^z \subseteq 3Q$, so
$$
\left(\prod_{i=1}^m \fint_{Q}|f_i| \right)\chi_{E_Q}(z) \leq \mathcal{M}(\vec{f})(z),
$$
and
$$
\frac{1}{W(3Q)}\int_{Q}\chi_F W \leq \frac{1}{W(Q^z)}\int_{Q^z}\chi_F W \leq M_W^c(\chi_F)(z).
$$

Since the sets $\{E_Q\}_{Q\in \mathcal S}$ are pairwise disjoint, and using H\"older's inequality with exponent $\frac{p}{\varepsilon}>1$,
$$
\begin{aligned}
II & \leq \left(\frac{3^n}{\eta}\Vert W \Vert_{A_r^{\mathcal R}}\right)^r \int_{\mathbb R^n} \mathcal{M}(\vec{f})^\varepsilon M_W^c(\chi_F) W \\ & \leq \left(\frac{3^n}{\eta}\Vert W \Vert_{A_r^{\mathcal R}}\right)^r \left \| \left(\frac{\mathcal{M}(\vec{f})}{v}\right) ^\varepsilon \right \|_{L^{p/\varepsilon,\infty}(w)}\left \| M_W^c(\chi_F) \right \|_{L^{(p/\varepsilon)',1}(w)} \\ & \leq \frac{p}{p-\varepsilon} \left(\frac{3^n}{\eta}\Vert W \Vert_{A_r^{\mathcal R}}\right)^r \left \| M_W^c \right \|_{L^{(p/\varepsilon)',1}(w)} w(F)^{1-\frac{\varepsilon}{p}} \left \| \frac{\mathcal{M}(\vec{f})}{v}\right \|_{L^{p,\infty}(w)}^\varepsilon.  
\end{aligned}
$$

Observe that for every measurable function $g$, $\Vert M_W ^c (g)\Vert_{L^\infty(w)} \leq \Vert g \Vert_{L^\infty(w)}$, and by standard arguments (see \cite[Theorem 7.1.9]{grafclas}), it is easy to show that $$\Vert M_W ^c (g) \Vert_{L^{1,\infty}(w)}\lesssim_n  [v^{-\varepsilon}]_{RH_\infty(w)}\Vert g \Vert_{L^1(w)}.$$ 
In particular, and applying Marcinkiewicz's interpolation theorem (see \cite[Theorem 4.13]{BS}), we conclude that 
$$\left \| M_W^c \right \|_{L^{(p/\varepsilon)',1}(w)} \leq c_{n,p,\varepsilon}[v^{-\varepsilon}]_{RH_\infty(w)}^{1-\frac{\varepsilon}{p}}<\infty.$$

Combining the previous estimates, we obtain that
$$
\begin{aligned}
&\left \| \frac{\mathcal{A}_{\mathcal S}(\vec{f})}{v}\chi_F \right \|_{L^{\varepsilon}(w)}w(F)^{\frac{1}{p}-\frac{1}{\varepsilon}} \\ & \leq  \left(\frac{p}{p-\varepsilon} \left(\frac{3^n}{\eta}\Vert W \Vert_{A_r^{\mathcal R}}\right)^r c_{n,p,\varepsilon}[v^{-\varepsilon}]_{RH_\infty(w)}^{1-\frac{\varepsilon}{p}}\right)^{1/\varepsilon} \left \| \frac{\mathcal{M}(\vec{f})}{v}\right \|_{L^{p,\infty}(w)},
\end{aligned}
$$
and the desired result follows, with
$$C=\inf_{r\geq 1: \, W\in A_r^{\mathcal R}}\left(\frac{p}{p-\varepsilon} \left(\frac{3^n}{\eta}\Vert W \Vert_{A_r^{\mathcal R}}\right)^r c_{n,p,\varepsilon}[v^{-\varepsilon}]_{RH_\infty(w)}^{1-\frac{\varepsilon}{p}}\right)^{1/\varepsilon}.$$
\end{proof}

\begin{remark}
For $0<p\leq 1$, if we take $v$ such that $v^\delta \in A_\infty$ for some $\delta>0$, and $w=uv^p$, with $u\in A_1$, then the previous result can be established via an extrapolation argument (see \cite[Theorem 1.1]{ompe}). 
\end{remark}

Under the conditions that $0<p\leq 1$, and $w=uv^p$, we can find weights $u$ and $v$ that satisfy the hypotheses of Theorem 1.1 in \cite{ompe} but not the ones of Theorem~\ref{sparsemax}, and vice versa. If we take a non-constant weight $u\in A_1$, and $v=u^{-1/p}$, then $v\in RH_\infty\subseteq A_{\infty}$, and $uv^p=1$, but for every $0<\varepsilon \leq 1$ such that $\varepsilon <p$, we have that $v^{-\varepsilon} = u^{\varepsilon/p}\in A_1$, and since $u$ is non-constant, $v^{-\varepsilon} \not \in RH_{\infty}$. Similarly, if we take a non-constant weight $v\in A_1$, and $u=v^{-p}$, then $uv^p=1$, and for every $\varepsilon>0$, $uv^{p-\varepsilon}=v^{-\varepsilon} \in RH_{\infty}\subseteq A_{\infty}$, but $u\in RH_{\infty}$ and is non-constant, so $u\not \in A_1$ (see Lemma \ref{pesos}).

The previous examples show that, sometimes, some of the hypotheses of Theorem~\ref{sparsemax} may be redundant. Let us be more precise on this fact. If $w\in A_{\infty}$, and $wv^{-\varepsilon}$ is a weight, then $[v^{-\varepsilon}]_{RH_\infty (w)}<\infty$ implies that $wv^{-\varepsilon}\in A_\infty$. Indeed, given a cube $Q\subseteq \mathbb R^n$, and a measurable set $E\subseteq Q$, we have that
$$
\frac{wv^{-\varepsilon}(E)}{wv^{-\varepsilon}(Q)}=\frac{1}{wv^{-\varepsilon}(Q)}\int_Q \chi_E wv^{-\varepsilon} \leq \frac{w(E)}{wv^{-\varepsilon}(Q)}\Vert \chi_Q v^{-\varepsilon}\Vert_{L^{\infty}(w)} \leq [v^{-\varepsilon}]_{RH_\infty (w)}\frac{w(E)}{w(Q)},
$$
and since $w\in A_\infty$, there exist $\delta, C>0$ such that
$$
\frac{w(E)}{w(Q)} \leq C \left(\frac{|E|}{|Q|} \right)^\delta,
$$
so
$$
\frac{wv^{-\varepsilon}(E)}{wv^{-\varepsilon}(Q)} \leq C [v^{-\varepsilon}]_{RH_\infty (w)} \left(\frac{|E|}{|Q|} \right)^\delta,
$$
and hence, $wv^{-\varepsilon}\in A_\infty$ (see \cite{DMRO}).

The next application of Theorem~\ref{sawyer} follows from the combination of Theorems~\ref{msawyer} and \ref{sparsemax}, and gives us mixed restricted weak type bounds for multi-variable sparse operators that can also be deduced for other operators, such as multi-linear Calder\'on-Zygmund operators, using sparse domination techniques (see \cite{li1}). 

\begin{theorem}\label{czo}
Let $1\leq p_1,\dots, p_m <\infty $, and let $\frac{1}{p}=\frac{1}{p_1}+ \dots + \frac{1}{p_m}$. Let $w_1, \dots,w_m$ be weights, with $w_i \in A_{p_i}^{\mathcal R}$, $i=1,\dots,m$, and write $\nu_{\vec w}=w_1^{p/p_1}\dots w_m ^{p/p_m}$. Let $v$ be a weight such that $\nu_{\vec w}v^p$ is a weight, and $w_iv^{p_i}\in A_\infty$, $i=1,\dots,m$. Moreover, suppose that there exists $0<\varepsilon \leq 1$ such that $\varepsilon < p$, $\nu_{\vec w}v^{p-\varepsilon} \in A_\infty$, and $[v^{-\varepsilon}]_{RH_\infty(\nu_{\vec w}v^p)}<\infty$. Then, there exists a constant $C>0$ such that the inequality
$$
\left \| \frac{T (\vec f)}{v} \right \|_{L^{p,\infty}(\nu_{\vec w}v^p)} \leq C \prod_{i=1}^m \left \| f_i \right \|_{L^{p_i,1}(w_i)}
$$ 
holds for every vector of measurable functions $\vec f=(f_1,\dots,f_m)$, where $T$ is either a sparse operator of the form
\begin{equation}\label{sparseopa}
\mathcal A_{\mathcal S} (\vec f):= \sum_{Q\in \mathcal S}  \left( \prod_{i=1} ^m\fint_Q f_i \right)\chi_Q,
\end{equation}
where $\mathcal S$ is an $\eta$-sparse collection of dyadic cubes, or any operator that can be conveniently dominated by such sparse operators, like $m$-linear $\omega$-Calder\'on-Zygmund operators with $\omega$ satisfying the Dini condition.
\end{theorem}

\begin{remark}
In the case when $p_1=\dots=p_m=1$, and $T$ is a multi-linear Calder\'on-Zygmund operator, the previous result follows from Theorem 1.9 in \cite{kob}.
\end{remark}

In general, there are examples of weights that satisfy the hypotheses of Theorem~\ref{czo} apart from the constant weights. For instance, if $1\leq p_1,\dots,p_m \leq m'$, we can take $w_i=(Mh_i)^{(1-p_i)/m}$, with $h_i \in L^1_{loc}(\mathbb R^n)$, $i=1,\dots,m$, and $v=\nu_{\vec w}^{-1/p}$. Indeed, in virtue of Theorem 2.7 in \cite{cgs}, we have that $w_i \in A_{p_i}^{\mathcal R}$, $i=1,\dots,m$, and $w_i v^{p_i}=\left( \prod_{j\neq i} (Mh_j)^{1/p_j '}\right)^{p_i/m}\in A_1$. Observe that $\nu_{\vec w}v^p=1$, and $v=\left(\prod_{i=1}^m (Mh_i)^{1/p_i '}\right)^{1/m} \in A_1$, so for every $\varepsilon>0$, $\nu_{\vec w}v^{p-\varepsilon}=v^{-\varepsilon} \in RH_{\infty} \subseteq A_{\infty}$.

The last application that we provide of Theorem~\ref{sawyer} can be interpreted as a dual version of it, and generalizes \cite[Proposition 2.10]{CS}.

\begin{theorem}\label{dualsawyer}
Fix $p > 1$, and let $u$ and $v$ be weights such that $u\in A_{p}^{\mathcal R}$, $uv^p\in A_{\infty}$, and for some $0<\varepsilon \leq 1$, $uv^{p-\varepsilon}$ is a weight and  $[v^{-\varepsilon}]_{RH_{\infty}(uv^p)}<\infty$. Then, there exists a constant $C>0$ such that for every measurable function $f$, 
\begin{equation}\label{eqdualsawyer}
    \left \Vert \frac{M(fuv^{p-1})}{u}\right \Vert_{L^{p',\infty}(u)} \leq C \Vert f \Vert_{L^{p',1}(uv^p)}.
\end{equation}
\end{theorem}

\begin{proof}
It is known (see \cite{lerner}) that there exist a collection $\{\mathscr D_{\alpha}\}_{\alpha}$ of $2^n$ general dyadic grids in $\mathbb R^n$, and a collection $\{\mathcal S_ {\alpha} \}_{\alpha}$ of $\frac{1}{2}$-sparse families of cubes, with $\mathcal S_ {\alpha} \subseteq \mathscr D_{\alpha}$, such that for every measurable function $F$,
$$
MF \leq 2 \cdot 12^n \sum_{\alpha = 1}^{2^n} \mathcal A_{\mathcal S_{\alpha}} (|F|).
$$
Hence,
$$
\left \Vert \frac{M(fuv^{p-1})}{u}\right \Vert_{L^{p',\infty}(u)} \leq 2 \cdot 24^n \sum_{\alpha =1}^{2^n} \left \Vert \frac{\mathcal A_{\mathcal S_{\alpha}}(|f|uv^{p-1})}{u}\right \Vert_{L^{p',\infty}(u)}.
$$

By duality, and self-adjointness of $\mathcal A_{\mathcal S_{\alpha}}$, and in virtue of H\"older's inequality, we have that
\begin{align*}
    \left \Vert \frac{\mathcal A_{\mathcal S_{\alpha}}(|f|uv^{p-1})}{u}\right \Vert_{L^{p',\infty}(u)} & \leq p \sup_{\Vert g \Vert_{L^{p,1}(u)}\leq 1} \left \{\int_{\mathbb R^n} \mathcal A_{\mathcal S_{\alpha}}(|f|uv^{p-1}) |g| \right \} \\ & = p \sup_{\Vert g \Vert_{L^{p,1}(u)}\leq 1} \left \{\int_{\mathbb R^n} |f|uv^{p-1} \mathcal A_{\mathcal S_{\alpha}}(|g|) \right \} \\ & \leq p \sup_{\Vert g \Vert_{L^{p,1}(u)}\leq 1} \left \{\left \Vert \frac{\mathcal A_{\mathcal S_{\alpha}}(|g|)}{v} \right \Vert_{L^{p,\infty}(uv^p)} \right \} \Vert f \Vert_{L^{p',1}(uv^p)},
\end{align*}
and the desired result follows from Theorem~\ref{sparsemax} and Theorem~\ref{sawyer}.
\end{proof}

\begin{remark}
It is clear from the previous proof that Theorem~\ref{dualsawyer} is also true for operators that can be conveniently dominated by sparse operators $A_{\mathcal S_{\alpha}}$. Even more, for a self-adjoint operator $T$, and by duality, an inequality like \eqref{eqdualsawyer} follows immediately from an inequality like \eqref{eqsawyer}, with $T$ in place of $M$.
\end{remark}

Note that for $p>1$, if $u\in A_p$, and $v$ is a weight, then for every measurable function $f$,
\begin{align*}
    \left \Vert \frac{M(fuv^{p-1})}{u}\right \Vert_{L^{p',\infty}(u)} & \leq \left \Vert \frac{M(fuv^{p-1})}{u}\right \Vert_{L^{p'}(u)} = \Vert M(fuv^{p-1}) \Vert_{L^{p'}(u^{1-p'})} \\ & \lesssim  [u]_{A_p} \Vert f\Vert_{L^{p'}(uv^{p})} \lesssim [u]_{A_p} \Vert f\Vert_{L^{p',1}(uv^{p})}.
\end{align*}
Hence, we obtain the conclusion of Theorem~\ref{dualsawyer} without assuming that for some $0<\varepsilon \leq 1$, $[v^{-\varepsilon}]_{RH_{\infty}(uv^p)}<\infty$. We would like to prove Theorem~\ref{dualsawyer} without this technical hypothesis, but unfortunately, at the time of writing, we don't know how to do it.

Observe that if $v=1$, then in Theorem~\ref{dualsawyer} we can take $\varepsilon = 1$, and
$
C= C_{n,p} [u]_{A_p^{\mathcal R}}^{p+1},
$
and the dependence on $u$ of the constant $C$ is explicit, although the exponent $p+1$ may not be sharp. Also, by testing on characteristic functions and using Kolmogorov's inequalities, we see that the condition that $u\in A_p^{\mathcal R}$ is necessary. This argument yields a new characterization of $A_p^{\mathcal R}$ weights and refines \cite[Proposition 2.10]{CS}.

\begin{theorem}
Fix $p > 1$, and let $u$ be a weight. If $u\in A_p^{\mathcal R}$, then for every measurable function $f$, 
\begin{equation*}
    \left \Vert \frac{M(fu)}{u}\right \Vert_{L^{p',\infty}(u)} \lesssim_{n,p} [u]_{A_p^{\mathcal R}}^{p+1} \Vert f \Vert_{L^{p',1}(u)}.
\end{equation*}
Moreover, if such an inequality holds for some constant $C>0$, then $u\in A_p^{\mathcal R}$, and $[u]_{A_p^{\mathcal R}}\leq p' C$.
\end{theorem}


\section{Sawyer-type inequalities and multi-variable conditions on weights} \label{section5}

In \cite[Theorem 1.5]{kob}, Li, Ombrosi, and Picardi obtained an endpoint Sawyer-type inequality for the operator $\mathcal M$ involving $A_{\vec 1}$ weights. It is natural to ask if something similar can be done in the general restricted weak setting, establishing a result for $\mathcal M$ like the one in Theorem~\ref{msawyer}, but assuming a multi-variable condition on the tuple of weights involved instead of imposing an individual condition on each weight. In this section we study this question.

In view of Theorem~\ref{msawyer} for $v=1$, it is reasonable to begin with the characterization of the weights for which the operator $\mathcal M$ and its centered version $\mathcal M^c$ are bounded in the restricted weak setting. This will give us the appropriate multi-variable condition on the weights. We use ideas from \cite[Section 3]{CHK} and \cite[Theorem 7.1.9]{grafclas}.

\begin{theorem}\label{frac}
Let $1\leq p_1,\dots, p_m <\infty $, and $\frac{1}{p}=\frac{1}{p_1} + \dots + \frac{1}{p_m}$. Let $w_1, \dots,w_m$, and $\nu$ be weights. The inequality 
\begin{equation}\label{eqfrach}
\Vert \mathcal M (\vec f) \Vert_{L^{p,\infty}(\nu)} \leq C \prod_{i=1}^m \Vert f_i \Vert_{L^{p_i,1}(w_i)} 
\end{equation} 
holds for every vector of measurable functions $\vec f$ if, and only if
\begin{equation}\label{eqfracc}
[\vec w, \nu]_{A_{\vec P}^{\mathcal R}} :=  \sup_{Q} \nu(Q)^{1/p} \prod_{i=1} ^m \frac{\Vert \chi_Q w_i ^{-1} \Vert_{L^{p_i',\infty}(w_i)}}{|Q|} < \infty.
\end{equation}
\end{theorem}

\begin{proof}
First, recall that by \cite[Theorem 1.4.16.(v)]{grafclas} (see also \cite[Theorem 4.4]{barza}), we have that
\begin{equation*}
\frac{1}{p_i}\left \| g \right \|_{L^{p'_i,\infty}(w_i)} \leq \sup \left \{\int_{\mathbb R^n} |fg|w_i : \left \| f \right \|_{L^{p_i,1}(w_i)} \leq 1 \right\} \leq \left \| g \right \|_{L^{p'_i,\infty}(w_i)}.
\end{equation*}

Now, fix a cube $Q$, and $\gamma > 1$, and for $i=1,\dots,m$, choose a non-negative function $f_i$ such that $\left \| f_i \right \|_{L^{p_i,1}(w_i)} \leq 1$ and
\begin{equation}\label{eqfrac1}
\int _Q f_i = \int_{\mathbb R^n} f_i (\chi_Q w_i^{-1}) w_i \geq \frac{1}{\gamma p_i} \Vert \chi_Q w_i ^{-1} \Vert_{L^{p'_i,\infty}(w_i)}.
\end{equation}
Since 
\begin{equation*} \left( \prod_{i=1}^m \frac{1}{|Q|}\int_Q |f_i|\right)\chi_Q\leq \mathcal M (\vec f) ,\end{equation*}
the hypothesis (\ref{eqfrach}) and (\ref{eqfrac1}) imply that
\begin{equation*}
\nu(Q)^{1/p} \prod_{i=1}^m \frac{\Vert \chi_Q w_i ^{-1} \Vert_{L^{p'_i,\infty}(w_i)}}{|Q|} \leq \gamma^m p_1 \dots p_m C,
\end{equation*}
and hence, $[\vec w, \nu]_{A_{\vec P}^{\mathcal R}} \leq p_1 \dots p_m C<\infty$. 

For the converse, suppose that the quantity $[\vec w, \nu]_{A_{\vec P}^{\mathcal R}}  < \infty$. Observe that
\begin{equation*}\mathcal{M} ^c (\vec f) \leq \mathcal M (\vec f) \leq 2^{nm} \mathcal{M}^c (\vec f),\end{equation*} 
so it suffices to establish the result for the operator $\mathcal{M}^c$. 

If for some $i=1,\dots,m$, $\left \| f_i \right \|_{L^{p_i,1}(w_i)}=\infty$, then there is nothing to prove, so we may assume that $\left \| f_i \right \|_{L^{p_i,1}(w_i)}<\infty$ for every $i=1,\dots,m$. 
Fix $\lambda >0$, and let $E_{\lambda} := \{x\in \mathbb R^n : \mathcal {M}^c (\vec f)(x) > \lambda\}$. We first show that this set is open. If for some $i=1,\dots,m$, $f_i\not \in L^1_{loc}(\mathbb R^n)$, then $E_{\lambda} = \mathbb R^n$. Otherwise, observe that for any $r>0$, and $x\in \mathbb R^n$, the function
\begin{equation*}
x \longmapsto \prod_{i=1} ^m \frac{1}{|Q(x,r)|}\int_{Q(x,r)}|f_i|
\end{equation*}
is continuous. Indeed, if $x_n \rightarrow x_0$, then $|Q(x_n,r)| \rightarrow |Q(x_0,r)|$, and also $\int_{Q(x_n,r)}|f_i| \rightarrow \int_{Q(x_0,r)}|f_i|$ by Lebesgue's dominated convergence theorem. Since $ |Q(x_0,r)| \neq 0$, the result follows. This implies that $\mathcal{M}^c (\vec f)$ is the supremum of continuous functions and hence, it is lower semi-continuous, and the set $E_{\lambda}$ is open. 

Given a compact subset $K$ of $E_{\lambda}$, for any $x \in K$, select an open cube $Q_x$ centered at $x$ such that
\begin{equation*}
\prod_{i=1}^m \frac{1}{|Q_x|}\int_{Q_x}|f_i| > \lambda. 
\end{equation*}
In virtue of \cite[Lemma 7.1.10]{grafclas}\footnote{Corrected in \url{https://faculty.missouri.edu/~grafakosl/Books_Correction_Pages/CFA-Page509-511.pdf}.}, we find a subset $\{Q_j\}_{j=1}^N$ of $\{Q_x:x\in K\}$ such that $K \subseteq \bigcup_{j=1}^N Q_j$, and $\sum_{j=1}^N \chi_{Q_j}\leq 72^n$. Then, by H\"older's inequality for Lorentz spaces, (\ref{eqfracc}), discrete H\"older's inequality with exponents $\frac{p_i}{p}$, and \cite[Lemma 2.5]{CHK},
\begin{align*}
\nu(K) & \leq \sum_{j=1}^N \nu(Q_j) \leq \frac{1}{\lambda ^p} \sum_{j=1}^N \nu(Q_j) \left(\prod_{i=1}^m \frac{1}{|Q_j|} \int_{Q_j}|f_i| \right)^p \\ & \leq \frac{1}{\lambda ^p} \sum_{j=1}^N \nu(Q_j)  \prod_{i=1}^m |Q_j|^{-p} \Vert f_i \chi_{Q_j}\Vert_{L^{p_i,1}(w_i)}^p \Vert \chi_{Q_j} w_i ^{-1} \Vert_{L^{p'_i,\infty}(w_i)} ^p \\ & \leq \frac{[\vec w, \nu]_{A_{\vec P}^{\mathcal R}}^p}{\lambda ^p} \sum_{j=1}^N  \prod_{i=1}^m \Vert f_i \chi_{Q_j}\Vert_{L^{p_i,1}(w_i)}^p \\ & \leq \frac{[\vec w, \nu]_{A_{\vec P}^{\mathcal R}}^p}{\lambda ^p} \prod_{i=1}^m \left ( \sum_{j=1}^N   \Vert f_i \chi_{Q_j}\Vert_{L^{p_i,1}(w_i)}^{p_i} \right) ^{p/p_i}\\ & \leq 72^{n} \frac{[\vec w, \nu]_{A_{\vec P}^{\mathcal R}}^p}{\lambda ^p} \prod_{i=1}^m \Vert f_i \Vert_{L^{p_i,1}(w_i)}^p.
\end{align*} 
Taking the supremum over all compact subsets $K$ of $E_{\lambda}$ and using the inner regularity of $\nu(x)dx$, 
we obtain (\ref{eqfrach}) with constant 
\begin{equation*}C=2^{nm}72^{n/p}[\vec w, \nu]_{A_{\vec P}^{\mathcal R}}.\end{equation*}
\end{proof}

\begin{remark}\label{normbounds}
In fact, we have proved that 
\begin{equation*}
\frac{2^{-nm}}{\prod_{i=1}^m p_i} [\vec w, \nu]_{A_{\vec P}^{\mathcal R}} \leq \Vert \mathcal M^c \Vert_{\prod_{i=1}^m L^{p_i,1}(w_i)\rightarrow L^{p,\infty}(\nu)} \leq 72^{n/p}[\vec w, \nu]_{A_{\vec P}^{\mathcal R}}
\end{equation*}
and
\begin{equation*}
\frac{1}{\prod_{i=1}^m p_i} [\vec w, \nu]_{A_{\vec P}^{\mathcal R}} \leq \Vert \mathcal M \Vert_{\prod_{i=1}^m L^{p_i,1}(w_i)\rightarrow L^{p,\infty}(\nu)} \leq 2^{nm}72^{n/p}[\vec w, \nu]_{A_{\vec P}^{\mathcal R}}.
\end{equation*}
\end{remark}

\begin{remark}
Observe that if $\mathcal M$ is bounded as in (\ref{eqfrach}), then for every cube $Q$, if we choose $f_1=\dots=f_m = \chi_Q$, we get that
\begin{equation*}
\left( \fint_Q \nu \right)^{1/p} \leq p_1 \dots p_m C  \prod_{i=1}^m \left( \fint_Q w_i \right)^{1/p_i},
\end{equation*}
and Lebesgue's differentiation theorem implies that $\nu \lesssim \prod_{i=1}^m w_i^{p/p_i}$.
\end{remark}

In virtue of Theorem~\ref{frac}, we define the following class of weights. 

\begin{definition}
Let $1\leq p_1,\dots, p_m <\infty $, and $\frac{1}{p}=\frac{1}{p_1} + \dots + \frac{1}{p_m}$. Let $w_1, \dots,w_m$, and $\nu$ be weights. We say that $(w_1,\dots,w_m,\nu)$ belongs to the class $A_{\vec P}^{\mathcal R}$ if $[\vec w, \nu]_{A_{\vec P}^{\mathcal R}} < \infty$.  
\end{definition} 

The condition that defines the class of $A_{\vec P}^{\mathcal R}$ weights depends on their behavior on cubes, and has been obtained following the ideas of Chung, Hunt, and Kurtz (see \cite{CHK}). One can ask if it is possible to obtain a different condition, resembling the one obtained by Kerman and Torchinsky (see \cite{kt}). Our next theorem gives a positive answer to this question, recovering their results in the case when $m=1$ and $w_1=\nu$.  

\begin{theorem}\label{multiaprweights}
Let $1\leq p_1,\dots, p_m <\infty $, and $\frac{1}{p}=\frac{1}{p_1} + \dots + \frac{1}{p_m}$. Let $w_1, \dots,$ $w_m,$ and $\nu$ be weights. The following statements are equivalent:
\begin{enumerate}
\item[$(a)$] $\Vert \mathcal M (\vec f) \Vert_{L^{p,\infty}(\nu)} \leq C \prod_{i=1}^m \left \| f_i \right \|_{L^{p_i,1}(w_i)}$, for every $\vec f$.
\item[$(b)$] $\left \| \mathcal M (\vec \chi) \right \|_{L^{p,\infty}(\nu)} \leq c \prod_{i=1}^m w_i(E_i)^{1/p_i}$, for every $\vec \chi = (\chi_{E_1},\dots,\chi_{E_m})$.
\item[$(c)$] \begin{equation*}\Vert \vec w, \nu \Vert_{A_{\vec P}^{\mathcal R}} :=  \sup_Q \nu(Q)^{1/p} \prod_{i=1}^m \sup_{0<w_i(E_i)<\infty} \frac{|E_i\cap Q|}{|Q|}w_i(E_i)^{-1/p_i}<\infty.\end{equation*}
\item[$(d)$] $(w_1,\dots,w_m,\nu)\in A_{\vec P}^{\mathcal R}$.
\end{enumerate}
Moreover, if $(w_1,\dots,w_m,\nu)\in A_{\vec P}^{\mathcal R}$, and $\nu \in A_\infty$, then 
\begin{equation}\label{eqmultiaprweightsc}
T:L^{p_1,1}(w_1)\times \dots \times L^{p_m,1}(w_m) \longrightarrow L^{p,\infty}(\nu),
\end{equation}
where $T$ is either a sparse operator as in \eqref{sparseopa},
or any operator that can be conveniently dominated by such sparse operators, like $m$-linear $\omega$-Calder\'on-Zygmund operators with $\omega$ satisfying the Dini condition.
\end{theorem}

\begin{proof}
It is clear that (a) implies (b), and we have already proved in Theorem~\ref{frac} that (a) and (d) are equivalent. Let us show that (b) implies (c). Fix a cube $Q$ and measurable sets $E_i$, for $i=1,\dots, m$, with $0<w_i(E_i)<\infty$. Since
\begin{equation*}
\left(\prod_{i=1} ^m \frac{|E_i\cap Q|}{|Q|} \right) \chi_Q \leq \mathcal M (\vec \chi), 
\end{equation*}
we apply (b) to conclude that
\begin{equation*}
\nu(Q)^{1/p} \prod_{i=1} ^m \frac{|E_i\cap Q|}{|Q|} \leq c \prod_{i=1}^m w_i(E_i)^{1/p_i},
\end{equation*}
and hence, $\Vert \vec w, \nu \Vert_{A_{\vec P}^{\mathcal R}} \leq c < \infty$. 

We now prove that (c) is equivalent to (d). First, observe that for every $i=1,\dots,m$,
\begin{equation*}
\sup_{0<w_i(E_i)<\infty} \frac{|E_i\cap Q|}{w_i(E_i)^{1/p_i}} = \sup_{E_i\subseteq Q} \frac{|E_i|}{w_i(E_i)^{1/p_i}},
\end{equation*}
where the first supremum is taken over all measurable sets $E_i$ such that $0<w_i(E_i)<\infty$, and the second one is taken over all non-empty measurable sets $E_i\subseteq Q$. Now, in virtue of \cite[Lemma 2.8]{CHK} and 
Kolmogorov's inequalities, we have that 
\begin{equation*}
\Vert \chi_Q w_i ^{-1}\Vert _{L^{p_i ', \infty}(w_i)} \leq \sup_{E_i \subseteq Q} \frac{|E_i|}{w_i(E_i)^{1/p_i}}\leq p_i \Vert \chi_Q w_i ^{-1}\Vert _{L^{p_i ', \infty}(w_i)},
\end{equation*}
and hence, $[ \vec w, \nu ]_{A_{\vec P}^{\mathcal R}} \leq \Vert \vec w, \nu \Vert_{A_{\vec P}^{\mathcal R}} \leq p_1 \dots p_m [ \vec w, \nu ]_{A_{\vec P}^{\mathcal R}}$.

Note that a similar argument to the one in the proof of Theorem \ref{sparsemax} shows that for $0< \varepsilon \leq 1$ such that $\varepsilon <p$, and $r\geq 1$ such that $\nu \in A_r^{\mathcal R}$,
\begin{equation}\label{eqmultiaprweights1}
\Vert \mathcal M _{\mathcal S} (\vec f)\Vert_{L^{p,\infty}(\nu)} \leq \Vert \mathcal A_{\mathcal S} (|\vec f|)\Vert_{L^{p,\infty}(\nu)} \leq C_{\varepsilon,\eta,n,p,r} [\nu]_{A_r^{\mathcal R}} ^{r/\varepsilon} \Vert \mathcal M _{\mathcal S} (\vec f)\Vert_{L^{p,\infty}(\nu)},
\end{equation}
where
\begin{equation*}
\mathcal M_{\mathcal S}(\vec f):=  \sup_{Q\in \mathcal S} \left(\prod_{i=1} ^m \fint_Q |f_i| \right)\chi_Q,
\end{equation*}
and since $\mathcal S$ is a countable collection of dyadic cubes, the proof of Theorem \ref{frac} can be rewritten to show that
\begin{equation*}
\mathcal M_{\mathcal S}:L^{p_1,1}(w_1)\times \dots \times L^{p_m,1}(w_m) \longrightarrow L^{p,\infty}(\nu)
\end{equation*}
if, and only if 
\begin{equation*}
[\vec w,\nu]_{A_{\vec P, \mathcal S}^{\mathcal R}}:= \sup_{Q\in \mathcal S} \nu(Q)^{1/p} \prod_{i=1} ^m \frac{\Vert \chi_Q w_i ^{-1} \Vert_{L^{p_i',\infty}(w_i)}}{|Q|} < \infty,
\end{equation*}
which is true, since $[\vec w,\nu]_{A_{\vec P, \mathcal S}^{\mathcal R}} \leq [\vec w,\nu]_{A_{\vec P}^{\mathcal R}}<\infty$. Moreover,
\begin{equation*}
\frac{1}{\prod_{i=1}^m p_i} [\vec w,\nu]_{A_{\vec P, \mathcal S}^{\mathcal R}} \leq \Vert \mathcal M_{\mathcal S} \Vert_{\prod_{i=1}^m L^{p_i,1}(w_i)\rightarrow L^{p,\infty}(\nu)} \leq [\vec w,\nu]_{A_{\vec P, \mathcal S}^{\mathcal R}},
\end{equation*}
so (\ref{eqmultiaprweights1}) implies that
\begin{equation}\label{eqmultiaprweights2}
\frac{1}{\prod_{i=1}^m p_i} [\vec w,\nu]_{A_{\vec P, \mathcal S}^{\mathcal R}} \leq \Vert \mathcal A_{\mathcal S} \Vert_{\prod_{i=1}^m L^{p_i,1}(w_i)\rightarrow L^{p,\infty}(\nu)} \leq C_{\varepsilon,\eta,n,p,r} [\nu]_{A_r^{\mathcal R}} ^{r/\varepsilon}  [\vec w,\nu]_{A_{\vec P, \mathcal S}^{\mathcal R}}.
\end{equation}

Finally, in virtue of Theorem 1.2 and Proposition 3.1 in \cite{li1} (see also \cite[Theorem 3.1]{lernerak}), if $T$ is an $m$-linear $\omega$-Calder\'on-Zygmund operator with $\omega$ satisfying the Dini condition, then there exists a dimensional constant $0<\eta <1$ such that given compactly supported functions $f_i \in L^1(\mathbb R^n)$, $i=1,\dots,m$, there exists an $\eta$-sparse collection of dyadic cubes $\mathcal S$ such that   
\begin{equation*}
    |T(f_1,\dots,f_m)| \leq c_n C_T \mathcal A_{\mathcal S}(|\vec f|).
\end{equation*}
Hence, (\ref{eqmultiaprweightsc}) follows from (\ref{eqmultiaprweights2}) and the standard density argument in \cite[Exercise 1.4.17]{grafclas}. Moreover,
\begin{equation*}
    \Vert T \Vert_{\prod_{i=1}^m L^{p_i,1}(w_i)\rightarrow L^{p,\infty}(\nu)} \leq c_n C_T C_{\varepsilon,\eta,n,p,r} [\nu]_{A_r^{\mathcal R}} ^{r/\varepsilon}  [\vec w,\nu]_{A_{\vec P}^{\mathcal R}}.
\end{equation*}
\end{proof}

\begin{remark}
Given weights $w_1,\dots,w_m$, and $\nu=\prod_{i=1}^m w_i^{p/p_i}$, the equivalence between (b) and (c) in Theorem~\ref{multiaprweights} can be found in \cite{quico}. Moreover, if $p_1=\dots=p_m=1$, then the equivalence between (a) and (d) can be found in \cite{LOPTT}. Observe that if $\vec w \in A_{\vec P}$, then $(w_1,\dots,w_m,\nu_{\vec w})\in A_{\vec P}^{\mathcal R}$. 
In \cite{LOPTT}, strong and weak type bounds for $m$-linear Calder\'on-Zygmund operators were established for the first time for tuples of weights in $A_{\vec P}$. In \cite{lu}, these results were extended to $m$-linear $\omega$-Calder\'on-Zygmund operators with $\Vert \omega\Vert_{\text{Dini}}<\infty$.
\end{remark}

We can now state our main conjecture on Sawyer-type inequalities with $A_{\vec P}^{\mathcal R}$ weights, a complete multi-variable version of Theorem~\ref{sawyer} for $\mathcal M$.

\begin{conjecture}\label{conjmsawyer}
Let $1\leq p_1,\dots, p_m <\infty $, and let $\frac{1}{p}=\frac{1}{p_1}+ \dots + \frac{1}{p_m}$. 
Let $w_1, \dots,w_m$, and $\nu$ be weights, and suppose that $(w_1,\dots,w_m,\nu)\in A_{\vec P}^{\mathcal R}$. Let $v$ be a weight such that $\nu v^p\in A_{\infty}$. Then, there exists a constant $C>0$ such that the inequality
\begin{equation}\label{eqconjmsawyer}
\left \| \frac{\mathcal M (\vec f)}{v} \right \|_{L^{p,\infty}(\nu v^p)} \leq C \prod_{i=1}^m \left \| f_i \right \|_{L^{p_i,1}(w_i)}
\end{equation}
holds for every vector of measurable functions $\vec f=(f_1,\dots,f_m)$.
\end{conjecture}

\begin{remark}
This conjecture is true in the case when $p_1=\dots=p_m=1$ and $\nu = \nu_{\vec w}$, as shown in \cite[Theorem 1.5]{kob}. We don't know if the hypothesis that $\nu v^p\in A_{\infty}$ can be replaced by $v^\delta \in A_\infty$ for some $\delta>0$.
\end{remark}

In virtue of H\"older's inequality, if $\vec w \in \prod_{i=1}^m A_{p_i}^{\mathcal R}$, then $(w_1,\dots,w_m,\nu_{\vec w})\in A_{\vec P}^{\mathcal R}$, so this conjecture extends the result for $\mathcal M$ presented in Theorem~\ref{msawyer}. Also, combining such conjecture with Theorem~\ref{sparsemax}, we would get a generalization of Theorem~\ref{czo} in the line of \cite[Theorem 1.9]{kob}. 

As it happens in the one-dimensional case, the conclusion of Conjecture~\ref{conjmsawyer} is completely elementary if $1<p_1,\dots,p_m<\infty$, $\vec w \in A_{\vec P}$, and $\nu=\nu_{\vec w}$, since
\begin{align*}
\left \| \frac{\mathcal M (\vec f)}{v} \right \|_{L^{p,\infty}(\nu_{\vec w} v^p)} & \leq \left \| \frac{\mathcal M (\vec f)}{v} \right \|_{L^{p}(\nu_{\vec w} v^p)} =\| \mathcal M (\vec f) \|_{L^{p}(\nu_{\vec w})} \\ & \lesssim [\vec w]_{A_{\vec P}}^{\max\{\frac{p_1 '}{p},\dots,\frac{p_m '}{p}\}} \prod_{i=1}^m \Vert f_i \Vert_{L^{p_i}(w_i)}\lesssim [\vec w]_{A_{\vec P}}^{\max\{\frac{p_1 '}{p},\dots,\frac{p_m '}{p}\}} \prod_{i=1}^m \Vert f_i \Vert_{L^{p_i,1}(w_i)},
\end{align*}
where we have used the sharp estimates for $\mathcal M$ proved in \cite[Theorem 1.2]{limosu}.

In the general case, observe that for every $\theta >0$,
\begin{align*}
    \mathcal M (\vec f) & \leq [\vec w,\nu]_{A_{\vec P}^{\mathcal R}} \sup_Q \frac{\chi_Q}{\nu(Q)^{1/p}}\prod_{i=1}^m \Vert f_i \chi_Q \Vert_{L^{p_i,1}(w_i)}\\ & =[\vec w,\nu]_{A_{\vec P}^{\mathcal R}} \left(\sup_Q \frac{\chi_Q}{\nu(Q)^{\theta/p}}\prod_{i=1}^m \Vert f_i \chi_Q \Vert_{L^{p_i,1}(w_i)}^{\theta}\right)^{1/\theta}=:[\vec w,\nu]_{A_{\vec P}^{\mathcal R}}\mathcal N^{\theta}_{\vec w,\nu} (\vec f)^{1/\theta},
\end{align*}
and
$$
\left \| \frac{\mathcal M (\vec f)}{v} \right \|_{L^{p,\infty}(\nu v^p)} \leq [\vec w,\nu]_{A_{\vec P}^{\mathcal R}} 
\left \| \frac{\mathcal N^{\theta}_{\vec w,\nu} (\vec f)}{V} \right \|_{L^{p/\theta,\infty}(\nu V^{p/\theta})}^{1/\theta},
$$
with $V:=v^{\theta}$. We suspect that a wise choice of $\theta$ (maybe $\theta = p$ or $\theta = mp$) and the argument in the proof of Theorem 1.5 in \cite{kob} could lead to some advances towards our conjecture. This idea requires further investigation.


\section{Acknowledgements}

E. R. P. wants to express his sincere gratitude to David V. Cruz-Uribe for his help and efforts when we were trying to weaken the hypotheses of Theorem~\ref{sawyer}.

The authors thank M. J. Carro for her suggestions to improve this document. They also thank the reviewer for the feedback, which led to the release of Section~\ref{section5}.

\end{document}